\documentclass[11pt,letterpaper]{amsart}
\usepackage[T1]{fontenc}
\usepackage[utf8]{inputenc}
\usepackage{lmodern}
\usepackage{microtype}
\usepackage{amsmath,amscd,amssymb,amsthm,mathtools}
\usepackage{mathrsfs}
\usepackage{geometry}
\usepackage{enumitem}
\usepackage{xcolor}
\usepackage{hyperref}
\allowdisplaybreaks[2]
\numberwithin{equation}{section}
\hypersetup{
  colorlinks=true,
  linkcolor=blue!45!black,
  citecolor=blue!45!black,
  urlcolor=blue!45!black,
  pdftitle={Smooth Limits of Stable Manifolds},
  pdfsubject={Smooth Limits of Stable Manifolds}
}

\newtheorem{theorem}{Theorem}[section]
\newtheorem{lemma}[theorem]{Lemma}
\newtheorem{proposition}[theorem]{Proposition}
\newtheorem{corollary}[theorem]{Corollary}
\theoremstyle{definition}
\newtheorem{definition}[theorem]{Definition}

\newtheorem{remark}[theorem]{Remark}

\newtheorem{quest}[theorem]{Question}

\newcommand{\D}{\Delta}

\newcommand{\Q}{\mathbb Q}
\newcommand{\C}{\mathbb C}
\newcommand{\R}{\mathbb R}
\newcommand{\PP}{\mathbb P}
\newcommand{\OO}{\mathcal O}
\newcommand{\IC}{\operatorname{IC}}
\newcommand{\IH}{\operatorname{IH}}
\newcommand{\Ib}{\operatorname{Ib}}
\newcommand{\Bs}{\operatorname{Bs}}
\newcommand{\Proj}{\operatorname{Proj}}
\newcommand{\Exc}{\operatorname{Exc}}
\newcommand{\Supp}{\operatorname{Supp}}

\newcommand{\Aut}{\operatorname{Aut}}
\newcommand{\ord}{\operatorname{ord}}
\newcommand{\codim}{\operatorname{codim}}
\newcommand{\Gr}{\operatorname{Gr}}

\newcommand{\Perv}{\operatorname{Perv}}
\DeclareMathOperator{\Hilb}{Hilb}
\DeclareMathOperator{\Tot}{Tot}
\DeclareMathOperator{\Bl}{Bl}
\DeclareMathOperator{\Spec}{Spec}
\DeclareMathOperator{\pr}{pr}

\DeclareMathOperator{\Ric}{Ric}
\DeclareMathOperator{\Sing}{Sing}
\DeclareMathOperator{\Douady}{Douady}
\title{\bfseries Smooth Limits of Stable Manifolds}

\author{Mu-lin Li}
\address{School of Mathematics, Hunan University, China}
\email{mulin@hnu.edu.cn}

\thanks{\textbf{Keywords}: K\"ahler--Einstein metrics; smooth limits;
Moishezon manifolds; canonical models; intersection cohomology;
rational curves; elementary transformations.}
\thanks{\textbf{2020 Mathematics Subject Classification}:
Primary 32G05; Secondary 32J18, 14E30, 14F43, 53C55.}

\date{}
\begin{document}

\begin{abstract}
 Let $X$ be a compact complex manifold of dimension $n$, and suppose that $X$ is deformation equivalent to a stable manifold. Then $X$ is projective whenever $n\le3$. We also construct, for every $n\ge4$, smooth families of compact complex manifolds over the unit disc such that every fiber over the punctured disc is a stable manifold,  while the central fiber is a non-K\"ahler Moishezon manifold.
\end{abstract}

\maketitle

\section{Introduction}\label{sec:introduction}

Two compact complex manifolds $X$ and $X'$ are \emph{deformation equivalent} if there is a smooth proper morphism $f\colon\mathcal{X}\to B$ between connected complex manifolds together with points $b,b'\in B$ such that $X_b\cong X$ and $X_{b'}\cong X'$.

% Equivalently, if and only if there is a sequence of compact complex manifolds $\{X_i\}_{i=0}^k$ such that $X_0=X$, $X_k=X'$, and  a series of smooth proper morphisms
%\begin{equation}\nonumber
%\pi_i:\mathcal{X}_i\to\Delta,
%\end{equation}
%where $\Delta=\{t\in\C||t|<1\}\subset \C$ is the unit disc, and $X_{i-1}$, respectively $X_{i}$ are isomorphic to fibers of $\pi_i$.

Yau proposed the following question, concerning a compact complex manifold that is deformation equivalent to a K\"ahler--Einstein manifold of nonpositive scalar curvature.

\begin{quest}\cite[Chapter~IX, \S~IV, Problem~55]{SY94}\label{eq:question}
If a compact complex manifold $X$ can be deformed to a K\"ahler--Einstein manifold $X'$ of nonpositive scalar curvature, is $X$ K\"ahler?
\end{quest}

The question concerns the persistence of the K\"ahler property under deformation, not the persistence of an Einstein metric. Hironaka's example
\cite{Hironaka} shows that a smooth limit of projective manifolds can
be non-K\"ahler, so the Einstein condition imposes a substantive
additional constraint. For a K\"ahler form $\omega$ on $X'$ we adopt the convention
\[
 \Ric(\omega)=-\sqrt{-1}\,\partial\bar\partial\log\det(g_{j\bar k}),
 \qquad [\Ric(\omega)]=2\pi c_1(TX').
\]
A negative K\"ahler--Einstein metric can be rescaled so that
$\Ric(\omega)=-\omega$. This identity makes the canonical bundle
positive, hence ample by the Kodaira embedding theorem \cite[Section~1.2.C]{Lazarsfeld2004}. Conversely,
the existence theorems of Aubin \cite{Aubin1978} and Yau
\cite[Section~6, Theorem~5]{Yau1978} produce such a metric whenever the
canonical bundle is ample. The metric hypothesis is therefore
equivalent to canonical polarization. By a \emph{stable
manifold} we mean a smooth canonically polarized manifold.
Since ampleness of the canonical line bundle is a Zariski-open condition, there is a proper analytic subset $S\subset B$ such that the canonical line bundles of the fibers over $B\setminus S$ are ample. Choose a small disc $\Delta_b\subset B$ centered at $b$ with $S\cap(\Delta_b\setminus\{b\})=\emptyset$. The negative scalar curvature case of Question~\ref{eq:question} then reduces to the following problem.

Let
$\pi\colon\mathcal X\to\Delta$ be a proper holomorphic submersion
with connected compact fibers $X_t$ over the unit disc, and suppose that every fiber over the punctured disc $\Delta^*:=\Delta\setminus\{0\}$ is a stable manifold. Must $X_0$ be K\"ahler?

In this paper we give a complete,
dimension-sharp answer to the \emph{negative K\"ahler--Einstein
smooth-limit case of Yau's question}: the answer is affirmative in
complex dimension at most three and negative in every complex
dimension at least four.

 Our principal
positive result is the following.

\begin{theorem}[cf. Theorem \ref{thm:main}]\label{thm:main-ke}
Let $\pi\colon\mathcal X\to\Delta$ be a proper holomorphic submersion
with connected compact fibers of complex dimension $1\le n\le3$. Suppose
that for every $t\in\Delta^*$, the fiber $X_t$ is a stable manifold.
Then $X_0$ is a smooth projective manifold, hence a smooth K\"ahler
manifold. Moreover, $K_{X_0}$ is big and semiample.
\end{theorem}

The essential case is dimension three. More precisely, we prove that
the canonical morphism of the central fiber is crepant, that its fibers have dimension at
most one, and that it contracts every exceptional prime divisor onto a curve;
see Theorem~\ref{thm:main}. In dimension two, the central fiber is a K\"ahler surface whenever the fibers over the punctured disc are K\"ahler surfaces, by the
theory of classification of surfaces and Siu's result \cite{Siu} (see also \cite{Bu}, \cite{La}, \cite{LL24}).

The dimension bound is optimal. For $n\ge4$ we have the following theorem.

\begin{theorem}\label{thm:KE-main}
For every integer $n\ge4$, there is a proper holomorphic submersion
$\pi\colon\mathcal M\to\Delta$ with connected compact fibers of
complex dimension $n$ such that
\begin{enumerate}
\item $K_{M_t}$ is ample for every $t\ne0$;
\item $M_0$ is Moishezon and non-K\"ahler.
\end{enumerate}
\end{theorem}

Together, these two theorems settle the negative-curvature
smooth-limit problem in every dimension. The counterexamples give a
negative answer to Yau's question in dimensions at least four even when
every punctured fiber is assumed to be canonically polarized.

We briefly explain the proof of Theorem~\ref{thm:main-ke}. Riemann--Roch,
Kodaira vanishing on the punctured fibers, and semicontinuity first show
that $K_{X_0}$ is big, so $X_0$ is Moishezon. Invariance of plurigenera
\cite[Theorem~1.2(i)]{RaoTsai2022} and finite generation of canonical
rings \cite[Corollary~1.1.1(2),(3)]{BCHM2010} then give a flat
projective relative canonical model $Y\to\Delta$. The key step is a
comparison of second cohomology with intersection cohomology. On a
resolution of the central canonical map, the decomposition theorem
counts the contributions of stable fixed divisors and crepant
exceptional divisors. On the degeneration side, the shifted nearby
intersection complex is a perverse sheaf. Its Jordan--H\"older
filtration has exactly one full-support factor, the intersection
complex of the central canonical model. All other factors lie in the
singular locus, and only codimension-two supports can contribute in
degree two. Their coefficient local systems form filtrations of the
transverse ADE Milnor local systems. Comparing Milnor and resolution
monodromies modulo the Weyl groups bounds their invariant dimensions
by the number of crepant divisorial valuations with codimension-two
center. Constancy of the second Betti number in the original smooth
family then forces every stable fixed divisor to disappear. This
comparison requires no purity assumption on the limiting cohomology.

In dimension three, the remaining stable base locus has dimension at
most one. A dimension argument on the normalized relative graph of the
canonical map eliminates it, yielding a crepant canonical morphism.
The same cohomological count excludes divisors contracted to points.
Its contracted curves are smooth rational curves, and they deform to the
exceptional curves of a general transverse Du Val surface resolution.
An inverse-Cartan-matrix construction, invariant under Dynkin-diagram
monodromy, produces a line bundle of positive degree on every
contracted curve. Since the fibers of the canonical morphism have
dimension at most one, a suitable power of this line bundle, twisted
by the pullback of an ample line bundle on the canonical model, has
positive degree on every irreducible curve in $X_0$.
Koll\'ar's projectivity criterion
\cite[Corollary~5.1.5]{Kollar1991} then shows that $X_0$ is projective.

For Theorem~\ref{thm:KE-main}, we start with a pencil of smooth quartic
K3 surfaces whose central member contains exactly one line and whose
nearby members contain none. On the relative Hilbert square, the
secant-line morphism contracts the plane $\ell^{[2]}\simeq\PP^2$ in
the central fiber and is finite on the other fibers. A double cover
branched in a sufficiently positive divisor pulled back by this
morphism makes the punctured canonical bundles ample. The branch
divisor avoids the contracted plane, whose inverse image consists of
two disjoint planes. The cyclic-cover Lefschetz theorem
\cite[Proposition~1.11]{LanteriStruppa} shows that their line classes
are homologous. We perform an elementary transformation along only
one plane; an explicit local model verifies that the resulting total
space and the map to the disc remain smooth. On the central fiber,
this transformation reverses one line class and leaves the other
unchanged. The two resulting effective curves therefore represent nonzero
opposite homology classes, which rules out a K\"ahler form. The central
fiber remains Moishezon. Products with curves of genus at least two
give the examples in all higher dimensions.

Sections~\ref{sec:bigness}--\ref{sec:ADE} establish the
canonical-model and cohomological arguments in arbitrary dimension.
Section~\ref{sec:base-elimination} completes the threefold proof, and
Section~\ref{sec:quartic} constructs the counterexamples.

\medskip
\noindent
{\bf Acknowledgement:}  The authors would like to thank Prof. J. Kollar for his interest in the construction of the counterexamples.

\medskip
\noindent
\textbf{Declaration on the Use of AI}: The author used ChatGPT (GPT-6.0) for language polishing and editing. ChatGPT (GPT-6.0) also assisted the author with  the codimension-two ADE argument of Section~\ref{sec:ADE}, and the construction of the counterexamples in Section~\ref{sec:quartic}. All mathematical arguments and statements are the responsibility of the author.

\section{Preliminaries: intersection theory and cohomological conventions}\label{sec:preliminaries}

%\begin{setup}\label{setup:main}
%In Sections~\ref{sec:bigness}--\ref{sec:exclude-divisors}, let
%$\pi\colon\mathcal X\to\Delta$ be a proper holomorphic submersion with
%connected compact fibers $X_t$ of complex dimension $n\ge1$, and assume
%that $K_{X_t}$ is ample for every $t\ne0$. Write $X=X_0$.
%Section~\ref{sec:base-elimination} specializes to $n=3$; the families
%in Section~\ref{sec:quartic} are constructed separately.
%\end{setup}
%

All complex spaces are taken with their reduced structure unless a fiber or a base ideal is explicitly considered scheme-theoretically. By a \emph{modification} we mean a proper bimeromorphic morphism. Cohomology has rational coefficients unless otherwise indicated. The hypercohomology of a complex is written $\mathbb H^j$; when no confusion can arise, it is also written $H^j$.

\subsection{Divisors, curve degrees, and the projection formula}\label{subsec:intersection-numbers}

Let $M$ be a smooth compact complex manifold. A Cartier divisor determines a line bundle, and a $\Q$-Cartier divisor is a rational divisor admitting a positive Cartier multiple. If $C\subset M$ is a reduced irreducible compact curve with normalization $\nu\colon\widetilde C\to C$, define
\begin{equation}\label{eq:curve-degree}
 L\cdot C=\deg\nu^*(L|_C)
          =\langle c_1(L),[C]\rangle.
\end{equation}
This is an integer, whether or not $M$ is K\"ahler. For an effective one-cycle $Z=\sum_i m_iC_i$ we set $L\cdot Z=\sum_i m_i(L\cdot C_i)$. Intersection with a $\Q$-Cartier divisor is defined by dividing the degree of a Cartier multiple.

Let $f\colon V\to M$ be a proper morphism, and let $R\subset V$ be a reduced irreducible compact curve. The projection formula reads
\begin{equation}\label{eq:projection-formula}
 f^*L\cdot R=L\cdot f_*[R]
 =\begin{cases}
 d\,L\cdot C,& f(R)=C\text{ is a curve and }d=\deg(R\to C),\\
 0,& f(R)\text{ is a point}.
 \end{cases}
\end{equation}
In the first case one applies the degree formula to the induced finite map of normalizations; the second follows because a line bundle pulled back from a point is trivial.

If $D\ge0$ is an effective $\Q$-Cartier divisor and $C\not\subset\Supp D$, then
\begin{equation}\label{eq:effective-curve}
 D\cdot C\ge0.
\end{equation}
Indeed, a defining section of a Cartier multiple restricts to a nonzero section on $\widetilde C$ and has an effective zero divisor. We assert no such nonnegativity when $C$ is contained in $\Supp D$.

On a smooth compact $n$-fold, top intersections are defined by
\[
 L_1\cdots L_n=\int_M c_1(L_1)\wedge\cdots\wedge c_1(L_n),
\]
using any closed representatives of the Chern classes.

\subsection{Positivity, stable base loci, and Moishezon manifolds}\label{subsec:positivity}

A line bundle $L$ on a compact complex $n$-fold is \emph{big} if its Iitaka dimension is $n$, equivalently if
\[
 \limsup_{m\to\infty}\frac{h^0(M,mL)}{m^n}>0.
\]
A compact complex manifold is Moishezon if and only if it carries a big line bundle \cite[Lemma~2.2]{RaoTsai2022}; equivalently, it is bimeromorphic to a projective variety. Resolving such a bimeromorphic correspondence gives a smooth projective modification.

A line bundle is \emph{semiample} if some positive tensor power is generated by its global sections. Its stable base locus is
\begin{equation}\label{eq:stable-base-definition}
 \mathbf B(L)=\bigcap_{m\ge1}\Bs|mL|,
\end{equation}
where $\Bs|mL|=M$ if $H^0(M,mL)=0$. When the section ring is finitely generated, a sufficiently divisible Veronese is generated in degree one; for such a degree $m$, the reduced zero set of the base ideal of $|mL|$ is exactly $\mathbf B(L)$. A prime divisor contained in this set is called a \emph{stable fixed prime divisor}.

Fix a Hermitian form $\omega$ on $M$. We say that $L$ is \emph{analytically nef} if for every $\varepsilon>0$ there is a smooth Hermitian metric $h_\varepsilon$ on $L$ such that
\begin{equation}\label{eq:analytic-nef}
 \frac{\sqrt{-1}}{2\pi}\Theta(L,h_\varepsilon)\ge-\varepsilon\omega.
\end{equation}
This definition is independent of the chosen Hermitian form. Semiampleness gives a smooth semipositive metric: pull back a Fubini--Study metric to a globally generated power and take its positive root. In particular,
\begin{equation}\label{eq:positivity-implications}
 L\text{ semiample}\ \Longrightarrow\ L\text{ analytically nef}
 \ \Longrightarrow\ L\cdot C\ge0\text{ for every compact curve }C.
\end{equation}
The second implication follows by integrating \eqref{eq:analytic-nef} on the normalization of $C$ and then letting $\varepsilon$ tend to zero.

\subsection{Discrepancies and exceptional negativity}\label{subsec:discrepancies}

Let $Y$ be a normal variety with $K_Y$ $\Q$-Cartier, and let $q\colon W\to Y$ be a resolution. With compatible canonical divisors, write
\begin{equation}\label{eq:discrepancy-convention}
 K_W=q^*K_Y+\sum_E a(E;Y)E.
\end{equation}
The coefficient $a(E;Y)$ is the \emph{discrepancy}; the log discrepancy is $1+a(E;Y)$. We always use the former convention. Thus canonical singularities mean $a(E;Y)\ge0$ for all exceptional divisorial valuations, and a valuation is called \emph{crepant} when its discrepancy is zero. On a smooth variety every exceptional divisorial valuation has strictly positive discrepancy \cite[Corollary~2.31(2) and Definition~2.34]{KollarMori1998}.

We use the following form of the negativity lemma. Suppose that $q\colon W\to Y$ is projective and bimeromorphic, with $W$ smooth and $Y$ normal. If a $q$-exceptional real Cartier divisor $D$ is $q$-nef, then
\begin{equation}\label{eq:negativity-prelim}
 D\le0.
\end{equation}
Here $q$-exceptional means that each component maps to a subset of codimension at least two, and $q$-nef means $D\cdot C\ge0$ for every $q$-contracted curve. This is the usual negativity lemma \cite[Lemma~3.39]{KollarMori1998}; its local surface proof applies also to projective morphisms with complex analytic target. In particular, an exceptional divisor numerically trivial on all contracted curves is zero, by applying \eqref{eq:negativity-prelim} to both signs.

\subsection{Cohomology with supports and exceptional divisor classes}\label{subsec:support-cohomology}

For a closed analytic subset $A$ of a smooth compact complex $n$-fold $M$, Alexander--Lefschetz duality identifies cohomology with supports and Borel--Moore homology:
\begin{equation}\label{eq:support-duality}
 H_A^k(M,\Q)\simeq H^{\mathrm{BM}}_{2n-k}(A,\Q).
\end{equation}
This follows by dualizing Poincar\'e--Alexander--Lefschetz duality over $\Q$ \cite[Chapter~VI, Section~8, Theorem~8.3]{Bredon1993}. Here $A$ is compact and triangulable, so ordinary and Borel--Moore homology agree. If $\codim_M A\ge2$, then $A$ has real dimension at most $2n-4$, and therefore
\begin{equation}\label{eq:codim-two-support}
 H_A^2(M,\Q)=H_A^3(M,\Q)=0,
 \qquad H^2(M,\Q)\simeq H^2(M\setminus A,\Q).
\end{equation}
The last assertion follows from the localization sequence.

If $E=\bigcup_{i=1}^k E_i$ is a union of prime divisors, its top Borel--Moore homology is generated by the fundamental classes of the $E_i$. Accordingly, the kernel of
\[
 H^2(M,\Q)\longrightarrow H^2(M\setminus E,\Q)
\]
is generated by the divisor classes $[E_i]=c_1(\OO_M(E_i))$. This is a spanning assertion; independence must be proved separately.

\subsection{Intersection complexes and the decomposition theorem}\label{subsec:intersection-complexes}

Let $V$ be an irreducible complex algebraic variety of dimension $d$, and let $j\colon U\hookrightarrow V$ be a dense smooth open subset. We normalize the intersection complex to be perverse:
\begin{equation}\label{eq:IC-convention}
 \IC_V=j_{!*}\Q_U[d],\qquad
 \IH^k(V,\Q)=\mathbb H^{k-d}(V,\IC_V),\qquad
 \Ib_k(V)=\dim_\Q\IH^k(V,\Q).
\end{equation}
Here $j_{!*}$ denotes intermediate extension in the category of perverse sheaves. If $V$ is smooth, then $\IC_V=\Q_V[d]$ and intersection cohomology agrees with ordinary cohomology. The shift convention is
\[
 \mathcal H^j(\mathcal K[a])=\mathcal H^{j+a}(\mathcal K).
\]
For the intersection complex and intermediate extension, see \cite[Sections~2.1 and~2.7]{deCataldoMigliorini2009}; the original construction is in \cite{GoreskyMacPherson1983}.

By a local system $\mathcal L$ on a smooth connected stratum $S$ we mean
a locally constant sheaf of finite-dimensional $\Q$-vector spaces for
the classical (analytic) topology. For a base point $x\in S$, parallel
transport defines a representation
\[
 \rho_{\mathcal L}\colon\pi_1(S,x)\longrightarrow
 \operatorname{GL}(\mathcal L_x),
 \qquad
 H^0(S,\mathcal L)\simeq
 \mathcal L_x^{\rho_{\mathcal L}(\pi_1(S,x))}.
\]
Thus the dimension of the space of global sections is the dimension
of the monodromy-invariant subspace, not necessarily the rank.

Let $S$ be a smooth connected locally closed stratum of complex
dimension $s$, and put $Z=\overline S$. Write
\[
 j_S\colon S\hookrightarrow Z,
 \qquad i_Z\colon Z\hookrightarrow V.
\]
For a local system $\mathcal L$ on $S$, define
\[
 \IC_Z(\mathcal L):=(j_S)_{!*}(\mathcal L[s]).
\]
The shift uses the dimension of $S$, not that of the ambient variety.
 The ordinary cohomology sheaves satisfy
\[
 \mathcal H^k(\IC_Z(\mathcal L))=0\quad(k<-s),
 \qquad
 \mathcal H^{-s}(\IC_Z(\mathcal L))\simeq(j_S)_*\mathcal L,
\]
where $(j_S)_*$ in the second formula is the ordinary sheaf direct
image. The hypercohomology spectral sequence therefore gives
\begin{equation}\label{eq:IC-lowest}
 \mathbb H^{-s}(Z,\IC_Z(\mathcal L))
 \simeq H^0(Z,(j_S)_*\mathcal L)
 =H^0(S,\mathcal L).
\end{equation}
Choose a stratification with respect to which $\IC_Z(\mathcal L)$ is
constructible. For a boundary stratum $S'\subset Z\setminus S$ of
dimension $s'$, intermediate extension gives the strict-support vanishing
\begin{equation}\label{eq:IC-boundary}
 \mathcal H^j\bigl(\IC_Z(\mathcal L)|_{S'}\bigr)=0
 \qquad(j\ge-s').
\end{equation}
These are ordinary cohomology-sheaf statements, not a convention for
perverse degrees; see \cite[Section~2.7]{deCataldoMigliorini2009}.

For a projective resolution $q\colon W\to V$ with $W$ smooth of
dimension $d$, put $\mathcal K=Rq_*\Q_W[d]$ and choose an adapted
algebraic Whitney stratification. The decomposition theorem makes
each perverse cohomology sheaf ${}^p\mathcal H^a(\mathcal K)$
semisimple. For $Z=\overline S$, let $\mathscr P_{Z,a}$ be the sum
of its simple direct summands with strict support $Z$. After replacing
$S$ by a smaller dense Zariski open stratum in $Z$, define the coefficient
local system by
\[
 \mathcal L_{S,a}:=
 \mathcal H^{-s}(\mathscr P_{Z,a}|_S),
 \qquad
 \mathscr P_{Z,a}|_S\simeq\mathcal L_{S,a}[s].
\]
It is zero if there is no such summand. These local systems are
semisimple by the decomposition theorem. The decomposition theorem gives a
noncanonical direct-sum decomposition
\begin{equation}\label{eq:DT-prelim}
 Rq_*\Q_W[d]\simeq
 \IC_V\oplus\bigoplus_{S,a}\IC_{\overline S}(\mathcal L_{S,a})[-a].
\end{equation}
Here the sum runs over proper supports. Relative hard Lefschetz pairs
the perverse shifts $a$ and $-a$ on each strict support, with the
appropriate Tate twist
\cite[Theorem~2.1.1(a)--(c) and Remark~2.1.3]{deCataldoMigliorini2005}.

The distinction between support dimension and perverse shift is particularly useful in degree two. A factor $\IC_{\overline S}(\mathcal L)[-a]$ can contribute to $H^2(W,\Q)$ only through
\[
 \mathbb H^{2-d-a}(\overline S,\IC_{\overline S}(\mathcal L)).
\]
The lower bound on the ordinary degrees of the intersection complex therefore imposes $2-d-a\ge-s$.

\subsection{Nearby cycles and finite-length perverse sheaves}\label{subsec:nearby-prelim}

For a holomorphic map $g\colon Y\to\Delta$, let $R\psi_g$ denote
the nearby-cycle functor. We use the perverse normalization
\begin{equation}\label{eq:perverse-nearby-convention}
 {}^p\!\psi_g=R\psi_g[-1],
\end{equation}
which is exact for the perverse $t$-structure
\cite[Theorem~5.5.3(2)]{deCataldoMigliorini2009}. All perverse
sheaves below have rational coefficients and are constructible for a
finite complex analytic stratification.

On a compact complex algebraic variety this category has finite
length, and its simple objects are the intersection complexes
$\IC_Z(\mathcal L)$, where $Z$ is irreducible and $\mathcal L$ is an
irreducible local system on a dense smooth open subset of $Z$
\cite[Section~2.3]{deCataldoMigliorini2009}. Thus every object
$\mathcal P$ has a finite Jordan--H\"older filtration
\[
 0=\mathcal P_0\subset\mathcal P_1\subset\cdots
 \subset\mathcal P_N=\mathcal P
\]
with simple successive quotients. This is a filtration in the
abelian category of perverse sheaves; no splitting is asserted.
For a short exact sequence in this category, the associated
hypercohomology long exact sequence gives, in every degree $k$,
\[
 \dim\mathbb H^k(Y_0,\mathcal P_i)
 \le\dim\mathbb H^k(Y_0,\mathcal P_{i-1})
   +\dim\mathbb H^k(Y_0,\mathcal P_i/\mathcal P_{i-1}).
\]
Restriction to an open subset is exact for the perverse
$t$-structure. Restriction to a stratum will instead always mean
ordinary derived restriction.

Assume now that $Y$ is normal of pure dimension $n+1$ and that
$g$ is proper and smooth over $\Delta^*$. Then
\[
 \mathcal J:={}^p\!\psi_g\IC_Y
 \in\Perv_{\Q}(Y_0).
\]
Nearby cycles depend only on the restriction to $Y_{\Delta^*}$.
Since this open set is smooth,
$\IC_Y|_{Y_{\Delta^*}}=\Q_{Y_{\Delta^*}}[n+1]$, and hence
\begin{equation}\label{eq:nearby-degree-prelim}
 \mathcal J\simeq R\psi_g\Q_Y[n],\qquad
 \mathbb H^{2-n}(Y_0,\mathcal J)\simeq H^2(Y_t,\Q).
\end{equation}
The second identification uses proper base change for nearby cycles,
a sufficiently small $t\ne0$, and a choice of its lift to the
universal cover of $\Delta^*$; see
\cite[Section~5.5]{deCataldoMigliorini2009}. In particular, its
left-hand side has dimension $b_2(Y_t)$. These are statements about
rational vector spaces; no limiting mixed Hodge structure is needed.

\section{Central bigness and the canonical fixed part}\label{sec:bigness}

The first reduction uses only the positivity of the punctured canonical bundles and the topology of the smooth family.

\begin{proposition}\label{prop:bigness}
Let
$\pi\colon\mathcal X\to\Delta$ be a proper holomorphic submersion
with connected compact fibers of dimension $n\ge1$ such that
$K_{X_t}$ is ample for $t\ne0$. Then $K_{X_0}$ is big and $X_0$ is Moishezon.
\end{proposition}

\begin{proof}
For $t\ne0$, ampleness of $K_{X_t}$ implies projectivity, so the integer
\[
 v:=\int_{X_t}c_1(K_{X_t})^n
\]
is positive. By Ehresmann's theorem and the relative tangent bundle, all Chern numbers of the fibers are constant; hence $v$ is independent of $t$. Hirzebruch--Riemann--Roch gives a common polynomial
\[
 P(a):=\chi(X_t,aK_{X_t})=\frac{v}{n!}a^n+O(a^{n-1}).
\]
For $a\ge2$ and $t\ne0$, Kodaira vanishing, applied to $K_{X_t}+(a-1)K_{X_t}$, gives $H^j(X_t,aK_{X_t})=0$ for $j>0$ \cite[Theorem~4.2.1]{Lazarsfeld2004}. Consequently
\[
 h^0(X_t,aK_{X_t})=P(a).
\]
For each fixed $a$, upper semicontinuity \cite[Section~7, Satz~3]{Grauert1960} yields
\[
 h^0(X_0,aK_{X_0})\ge P(a).
\]
This is a separate semicontinuity statement for each fixed $a$; no neighborhood uniform in $a$ is needed. The positive growth of the leading term $a^n$ proves that $K_{X_0}$ is big, and hence $X_0$ is Moishezon by Section~\ref{subsec:positivity}.
\end{proof}

Let now $M$ be a smooth compact Moishezon manifold with $K_M$ big. A smooth projective modification has the same canonical ring as $M$: pluricanonical forms descend across the modification, because they agree on the isomorphism locus and extend across a subset of codimension at least two in the smooth target. By finite generation of the canonical ring for smooth projective varieties of general type \cite[Corollary~1.1.1(2),(3)]{BCHM2010},
\[
 V:=\Proj R(M,K_M)
\]
is a projective canonical model with canonical singularities and ample $\Q$-Cartier canonical divisor.

Choose a sufficiently divisible integer $m$ such that $mK_V$ is very ample and the $m$-th Veronese of the canonical ring is generated in degree one. Let
\[
 p\colon W\to M,\qquad q\colon W\to V
\]
be a smooth projective common resolution, and choose canonical divisors compatibly.

\begin{lemma}\label{lem:fixed-part}
There is an effective $q$-exceptional $\Q$-divisor $F$ such that
\begin{equation}\label{eq:fixed-part}
 K_W=p^*K_M+E_p=q^*K_V+E_q,
 \qquad
 p^*K_M=q^*K_V+F,
 \qquad 0\le F\le E_q.
\end{equation}
If $\mathfrak b$ is the base ideal of $|mK_M|$, then after choosing the common resolution to principalize $\mathfrak b$,
\begin{equation}\label{eq:base-ideal}
 \mathfrak b\OO_W=\OO_W(-mF),
 \qquad
 V(\mathfrak b)_{\mathrm{red}}=\mathbf B(K_M).
\end{equation}
Moreover,
\[
 F=0
 \quad\Longleftrightarrow\quad
 K_M\text{ is semiample}
 \quad\Longleftrightarrow\quad
 K_M\text{ is analytically nef}.
\]
\end{lemma}

\begin{proof}
Since $M$ is smooth and $V$ has canonical singularities, the compatible
canonical divisors satisfy
\[
 K_W=p^*K_M+E_p=q^*K_V+E_q,
 \qquad E_p\ge0,\quad E_q\ge0,
\]
where $E_p$ is $p$-exceptional and $E_q$ is $q$-exceptional. Set
$F:=E_q-E_p$, so that $p^*K_M=q^*K_V+F$. We must prove that $F$ is
effective.

The natural identification of meromorphic pluricanonical forms induces
\[
 H^0(M,mK_M)\simeq H^0(W,mK_W)\simeq H^0(V,mK_V).
\]
For $p$, pluricanonical forms descend on the isomorphism locus and extend
across its complement of codimension at least two in the smooth manifold
$M$. For $q$, canonical singularities ensure that pluricanonical forms
pull back holomorphically to $W$; conversely, they descend and extend on
$V$ by normality, since $mK_V$ is Cartier.

For corresponding nonzero sections $s_M$, $s_W$, and $s_V$, let
$D_M$, $D_W$, and $D_V$ denote their effective zero divisors. The
discrepancy formulas give actual divisor equalities
\[
 D_W=p^*D_M+mE_p=q^*D_V+mE_q,
 \qquad p^*D_M=q^*D_V+mF.
\]
These equalities are established before any assertion that $F\ge0$.
Fix a prime divisor $G\subset W$ and choose a point $y\in q(G)$.
Since $mK_V$ is globally generated, there is a section $s_V$ with
$s_V(y)\ne0$. Its pullback, viewed as a section of
$q^*\OO_V(mK_V)$, does not vanish identically along $G$. Hence
\[
 \operatorname{coeff}_G(q^*D_V)=0.
\]
The corresponding $D_M$ is effective and Cartier, so $p^*D_M$ is
effective. Comparing coefficients in the preceding divisor equality
therefore yields
\[
 m\operatorname{coeff}_G(F)
 =\operatorname{coeff}_G(p^*D_M)\ge0.
\]
The section may depend on $G$. Since $G$ was arbitrary, $F\ge0$.
Here the nonvanishing statement concerns the pullback line bundle:
the corresponding section of $mK_W$ has the additional zero divisor
$mE_q$. Finally, $E_p\ge0$ gives $0\le F\le E_q$, and therefore $F$
is $q$-exceptional.

We next identify the base ideal. Under the compatible isomorphism
\[
 p^*\OO_M(mK_M)
 \simeq q^*\OO_V(mK_V)\otimes\OO_W(mF),
\]
the section $p^*s_M$ corresponds to $q^*s_V\otimes s_{mF}$, where
$s_{mF}$ is the canonical section of the effective divisor $mF$.
Because the sections $q^*s_V$ generate $q^*\OO_V(mK_V)$ everywhere,
the image of the pulled-back evaluation map is exactly
\[
 p^*\OO_M(mK_M)\otimes\OO_W(-mF).
\]
By the definition of the base ideal, this proves
$\mathfrak b\OO_W=\OO_W(-mF)$, in particular on a common resolution
principalizing $\mathfrak b$. Equivalently,
\[
 p^*|mK_M|=q^*|mK_V|+mF,
\]
with the first term on the right basepoint-free and $mF$ the fixed
divisor. Generation of the $m$-th Veronese in degree one implies that
every section of $kmK_M$ is a sum of products of $k$ sections of
$mK_M$. Thus, for every $k\ge1$,
\[
 \Bs|kmK_M|=\Bs|mK_M|.
\]
The intersection defining the stable base locus can be taken over
multiples of $m$: a section of any positive multiple of $K_M$ that
does not vanish at a point has a power of degree divisible by $m$
with the same property. Consequently
\[
 \mathbf B(K_M)=\Bs|mK_M|=V(\mathfrak b)_{\mathrm{red}},
\]
as asserted in \eqref{eq:base-ideal}.

If $F=0$, the line bundle $p^*\OO_M(mK_M)$ is globally generated.
The identity $p_*\OO_W=\OO_M$ and the projection formula identify
its global sections with those of $\OO_M(mK_M)$. Since $p$ is
surjective, global generation descends to $M$, so $K_M$ is semiample.
Conversely, if $K_M$ is semiample, some multiple $kmK_M$ is globally
generated. The equality of base loci above then shows that $mK_M$
is globally generated. Thus $\mathfrak b=\OO_M$, and the base-ideal
identity gives $mF=0$, hence $F=0$.

A globally generated positive tensor power of $K_M$ carries the
pullback of a Fubini--Study metric. Taking the corresponding root
gives a smooth Hermitian metric on $K_M$ with semipositive curvature;
in particular, semiampleness implies analytic nefness.
Finally, suppose that $K_M$ is analytically nef. It has nonnegative
degree on every compact curve, and its pullback has the same
property. For every irreducible curve $C\subset W$ contracted by $q$,
\[
 F\cdot C=(p^*K_M-q^*K_V)\cdot C=p^*K_M\cdot C\ge0.
\]
Thus $F$ is $q$-nef. Since it is also $q$-exceptional, the negativity
lemma \eqref{eq:negativity-prelim} for the projective morphism $q$
gives $F\le0$. Together with $F\ge0$, this proves $F=0$ and completes
the equivalences.
\end{proof}

\section{Second cohomology of resolutions and intersection cohomology}\label{sec:ih}

We compare a smooth modification with its smooth target, and a projective resolution with its singular target.

\begin{lemma}\label{lem:smooth-modification}
Let $p\colon W\to M$ be a modification between smooth compact complex $n$-folds, with $W$ projective. Let $P_1,\dots,P_k$ be the $p$-exceptional prime divisors. Then
\begin{equation}\label{eq:smooth-modification}
 H^2(W,\Q)=p^*H^2(M,\Q)\oplus\bigoplus_{j=1}^k\Q[P_j].
\end{equation}
In particular,
\[
 b_2(W)=b_2(M)+k.
\]
\end{lemma}

\begin{proof}
Let $A=p(\Exc p)$ and $U=M\setminus A$. Since $M$ is smooth, $\codim_MA\ge2$. By \eqref{eq:support-duality}, Alexander--Lefschetz duality gives
\[
 H^2_A(M,\Q)=H^3_A(M,\Q)=0,
\]
because the corresponding Borel--Moore homology degrees are $2n-2$ and $2n-3$, while $A$ has real dimension at most $2n-4$. Hence $H^2(M,\Q)\simeq H^2(U,\Q)$.

The localization sequence on $W$ shows that the kernel of restriction to $p^{-1}(U)$ is generated by the divisor classes $[P_j]$. Thus every class in $H^2(W)$ is a sum of a pullback class and exceptional divisor classes. If $D=\sum a_jP_j$ satisfies $[D]=p^*\gamma$, then $D\cdot C=0$ for every $p$-contracted curve $C$. Hence both $D$ and $-D$ are $p$-nef and $p_*D=0$, and applying the negativity lemma to both signs gives $D=0$. Finally, $p^*$ is injective: the topological pushforward for this degree-one map between oriented manifolds satisfies $p_*p^*=\mathrm{id}$. The sum is therefore direct.
\end{proof}

\begin{lemma}\label{lem:resolution-ih}
Let $q\colon W\to V$ be a projective resolution of a normal irreducible projective $n$-fold, with $W$ smooth and projective. Let $e(q)$ denote the number of $q$-exceptional prime divisors. Then
\begin{equation}\label{eq:resolution-ih}
 b_2(W)=\Ib_2(V)+e(q).
\end{equation}
\end{lemma}

\begin{proof}
Choose an algebraic Whitney stratification adapted to $q$. Since $V$ is normal, every nontrivial strict support has dimension $s\le n-2$. By the decomposition theorem and semisimplicity, with the normalization of Section~\ref{subsec:intersection-complexes},
\begin{equation}\label{eq:decomposition-general}
 Rq_*\Q_W[n]
 \simeq
 \IC_V\oplus\bigoplus_{S,a}\IC_{\overline S}(\mathcal L_{S,a})[-a].
\end{equation}
The full-support summand is $\IC_V$, with zero perverse shift.

Fix a connected smooth stratum $S$ of dimension $s\le n-2$, and let $d_S$ be the dimension of a general fiber over $S$. Birationality gives
\[
 d_S+s\le n-1.
\]
At a general point of $S$, the summand with shift $[-a]$ contributes in ordinary stalk degree $a-s$. In $Rq_*\Q_W[n]$ this is fiber cohomology in degree $n+a-s$, hence
\[
 n+a-s\le2d_S\le2(n-1-s).
\]
Therefore
\begin{equation}\label{eq:perverse-bound-upper}
 a\le n-2-s.
\end{equation}
Verdier duality gives the opposite inequality
\begin{equation}\label{eq:perverse-bound-both}
 s+2-n\le a\le n-2-s.
\end{equation}

Now
\[
 H^2(W,\Q)=H^{2-n}(V,Rq_*\Q_W[n]).
\]
The complex $\IC_{\overline S}(\mathcal L)$ belongs to the ordinary truncation category $D^{\ge-s}$. Thus a shifted factor can contribute to $H^{2-n}$ only if
\[
 2-n-a\ge-s,
 \quad\text{i.e.}\quad
 a\le s+2-n.
\]
Together with \eqref{eq:perverse-bound-both}, this forces
\[
 a=s+2-n.
\]
By \eqref{eq:IC-lowest}, its contribution is the lowest cohomology group
\[
 H^{2-n-a}(V,\IC_{\overline S}(\mathcal L_{S,a}))
 =H^{-s}(V,\IC_{\overline S}(\mathcal L_{S,a}))
 =H^0(S,\mathcal L_{S,s+2-n}).
\]
Consequently
\begin{equation}\label{eq:H2-support-sum}
 b_2(W)=\Ib_2(V)+\sum_S h^0(S,\mathcal L_{S,s+2-n}).
\end{equation}

Relative hard Lefschetz identifies, on each strict support,
\[
 \mathcal L_{S,s+2-n}\simeq \mathcal L_{S,n-2-s}
\]
(up to a Tate twist, which is irrelevant for ranks and invariant subspaces). We claim that the latter is the fiber-cohomology local system
\begin{equation}\label{eq:top-fiber-system}
 T_S:=R^{2(n-1-s)}q_*\Q_W|_S.
\end{equation}
Indeed, the relevant ordinary stalk degree of $Rq_*\Q_W[n]$ is
\[
 n-2-2s.
\]
A larger strict support $T\supsetneq S$ of dimension $t>s$ has perverse shift $a_T\le n-2-t$ by \eqref{eq:perverse-bound-upper}. To enter this stalk degree, its intersection complex would have to contribute in ordinary degree at least
\[
 n-2-2s-a_T\ge t-2s>-s.
\]
The boundary vanishing \eqref{eq:IC-boundary} excludes these degrees. The full-support summand $\IC_V$ makes no contribution for the same reason, since $n-2-2s\ge-s$. Smaller supports do not meet a general point of $S$. This proves the asserted identification with \eqref{eq:top-fiber-system}.

Put $r=n-1-s$. If $d_S<r$, then $T_S=0$ and no exceptional prime
divisor maps onto $\overline S$. If $d_S=r$, the stalks of $T_S$ are
the top cohomology groups of the fibers. They have bases dual to the
$r$-dimensional irreducible components, and monodromy only permutes
these classes. A monodromy orbit of these components closes up to
exactly one exceptional prime divisor of $W$ mapping onto $S$.
Taking global sections means taking the invariant subspace of this
permutation representation. In either case,
\[
 h^0(S,T_S)=\#\{E\subset W:E\text{ is }q\text{-exceptional and }q(E)=\overline S\}.
\]
Summing over $S$ in \eqref{eq:H2-support-sum} gives \eqref{eq:resolution-ih}.
\end{proof}

\begin{theorem}
\label{thm:ih-count-general}
Let $M$ be a smooth compact Moishezon manifold with $K_M$ big,
and let $V$ be its canonical model. If $r$ is the number of stable
fixed prime divisors of $K_M$ and $c$ is the number of crepant prime
divisors on $M$ contracted by the canonical map, then
\begin{equation}\label{eq:ih-count-general}
 b_2(M)-\Ib_2(V)=r+c.
\end{equation}
\end{theorem}
\begin{proof}
Take a common smooth projective resolution
\[
 p\colon W\to M,
 \qquad
 q\colon W\to V.
\]
Let $P$ and $Q$ be the sets of $p$-exceptional and $q$-exceptional prime divisors. With the notation of Lemma~\ref{lem:fixed-part}, every $p$-exceptional divisor has positive discrepancy over the smooth manifold $M$; since $F\ge0$, it is also $q$-exceptional. Thus $P\subset Q$.

The divisors in $Q\setminus P$ are strict transforms of prime divisors on $M$ contracted by the canonical map. For such a divisor the discrepancy over $M$ is zero. Its discrepancy over $V$ is positive precisely when it is a stable fixed divisor, and zero precisely when it is crepant and not fixed. Therefore
\[
 |Q|-|P|=r+c.
\]
By Lemmas~\ref{lem:smooth-modification} and \ref{lem:resolution-ih},
\[
 b_2(M)+|P|=b_2(W)=\Ib_2(V)+|Q|.
\]
Subtracting proves \eqref{eq:ih-count-general}.
\end{proof}

\section{Relative canonical models and nearby intersection complexes}
\subsection{The relative canonical model}\label{sec:relative-model}

The cohomological argument requires a projective family whose central fiber is the canonical model of $X_0$, not merely a projective family bimeromorphic to $\mathcal X$ over the punctured disc. We construct this family directly from the pluricanonical algebra.

\begin{lemma}\label{lem:uniform-generation}
Let
$\pi\colon\mathcal X\to\Delta$ be a proper holomorphic submersion
with connected compact fibers of dimension $n\ge1$ such that
$K_{X_t}$ is ample for $t\ne0$. There is a positive integer $m_*$ such that, for every $t\ne0$, the line bundle $m_*K_{X_t}$ is very ample and its section ring is generated in degree one.
\end{lemma}

\begin{proof}
By the Riemann--Roch calculation of Proposition~\ref{prop:bigness}, the polarized manifolds $(X_t,K_{X_t})$, $t\ne0$, all have the same Hilbert polynomial $P$. Matsusaka's theorem gives an integer $a\ge2$, depending only on $P$, such that $aK_{X_t}$ is very ample for every $t\ne0$ \cite{Matsusaka1972}; see \cite[Remark~1.2.12]{Lazarsfeld2004} for the precise statement. Kodaira vanishing gives $h^0(X_t,aK_{X_t})=P(a)$, so the corresponding complete embeddings lie in a fixed projective space $\PP^N$, where $N=P(a)-1$, and have Hilbert polynomial $P(ak)$.

The Hilbert scheme of subschemes of $\PP^N$ with this polynomial is of finite type, and its universal ideal sheaf satisfies a uniform Castelnuovo--Mumford regularity bound. In particular, there is an integer $b\ge1$ such that
\[
 H^1(\PP^N,\mathcal I_{X_t}(d))=0
 \qquad(t\ne0,\ d\ge b).
\]
Here the ideals refer to the embeddings by $|aK_{X_t}|$. For the regularity facts and their behavior in bounded families, see \cite[Section~1.8]{Lazarsfeld2004}. It follows that the restriction maps
\[
 H^0(\PP^N,\OO_{\PP^N}(kb))
 \longrightarrow H^0(X_t,kabK_{X_t})
\]
are surjective for all $k\ge1$ and $t\ne0$. Multiplication of homogeneous polynomials gives a surjection
\[
 \operatorname{Sym}^k H^0(\PP^N,\OO_{\PP^N}(b))
 \longrightarrow H^0(\PP^N,\OO_{\PP^N}(kb)),
\]
and restricting it to $X_t$ proves that the full section ring of $abK_{X_t}$ is generated in degree one. Thus $m_*=ab$ has the asserted properties, and the same holds for any positive multiple of $m_*$.
\end{proof}

\begin{proposition}\label{prop:relative-model}
After shrinking $\Delta$, there is a sufficiently divisible integer $m\ge2$ such that the graded $\OO_\Delta$-algebra
\[
 \mathscr R^{(m)}=\bigoplus_{k\ge0}\pi_*\omega_{\mathcal X/\Delta}^{\otimes km}
\]
is generated in degree one and
\begin{equation}\label{eq:relative-proj}
 g\colon Y=\Proj_\Delta\mathscr R^{(m)}\longrightarrow\Delta
\end{equation}
is flat and projective. For every $t\in\Delta$,
\[
 Y_t\simeq\Proj R(X_t,K_{X_t}),
 \qquad \OO_Y(1)|_{Y_t}\simeq\OO_{Y_t}(mK_{Y_t}).
\]
Over $\Delta^*$, the canonical map induces an isomorphism of families
\[
 \mathcal X_{\Delta^*}\simeq Y_{\Delta^*}.
\]
\end{proposition}

\begin{proof}
Every fiber is Moishezon, by Proposition~\ref{prop:bigness} and the hypothesis on the punctured fibers. Invariance of plurigenera for smooth families of Moishezon manifolds \cite[Theorem~1.2(i)]{RaoTsai2022}, together with Grauert's base-change theorem \cite[Section~7, Satz~5]{Grauert1960}, implies that
\[
 E_d=\pi_*\omega_{\mathcal X/\Delta}^{\otimes d}
\]
is locally free and commutes with base change for every $d\ge0$. In particular, its fiber at $t$ is $H^0(X_t,dK_{X_t})$, and multiplication in these bundles agrees with fiberwise multiplication of pluricanonical sections.

By finite generation of the central canonical ring, choose $m_0$ such that its $m_0$-th Veronese is generated in degree one and $m_0K_{Y_0}$ is very ample on the central canonical model. Let $m\ge2$ be a common multiple of $m_0$ and the integer $m_*$ of Lemma~\ref{lem:uniform-generation}. For every $k\ge1$ the multiplication morphism
\begin{equation}\label{eq:relative-multiplication}
 \operatorname{Sym}^k E_m\longrightarrow E_{km}
\end{equation}
is surjective on every fiber, including $t=0$. Its coherent cokernel is therefore zero by Nakayama's lemma. Thus $\mathscr R^{(m)}$ is a quotient of $\operatorname{Sym}^{\bullet}E_m$, and its relative Proj is projective over $\Delta$. Note that $m$ was chosen uniformly before applying Nakayama's lemma; no intersection of infinitely many shrinking neighborhoods is involved.

Base change in every degree identifies the fiber algebra with
\[
 \mathscr R^{(m)}\otimes\C(t)
 =\bigoplus_{k\ge0}H^0(X_t,kmK_{X_t}).
\]
Relative Proj commutes with this base change, so $Y_t$ is the canonical model of $X_t$, with the asserted tautological line bundle. To check flatness, work at a point of the disc. Every graded piece of $\mathscr R^{(m)}$ is free over the local discrete valuation ring. Homogeneous localizations, and then their degree-zero parts, are torsion-free over this ring, hence flat. These rings give the standard affine charts of the relative Proj, which proves that $g$ is flat.

For $t\ne0$, ampleness of $K_{X_t}$ identifies $X_t$ with the Proj of its complete canonical ring. The relative evaluation map defined by $E_m$ is a morphism over $\Delta^*$ and restricts on each fiber to the corresponding pluricanonical embedding. Its image is $Y_{\Delta^*}$, and it is an isomorphism there. This proves the final assertion.
\end{proof}

\begin{lemma}\label{lem:model-normality}
The total space $Y$ of Proposition~\ref{prop:relative-model} is normal
and of pure dimension $n+1$. Moreover, for $t\ne0$,
\[
 b_2(X_0)=b_2(Y_t).
\]
\end{lemma}

\begin{proof}
The fibers of $g$ are normal. Flatness over the smooth
one-dimensional base, together with the $S_2$ property of the fibers,
gives the $S_2$ property for $Y$. Regularity in codimension one
follows from normality of the generic fiber and reducedness at the
generic point of the central fiber, so $Y$ is normal by Serre's
criterion. Flatness and the pure dimension $n$ of the fibers give
pure dimension $n+1$ for $Y$.

Ehresmann's theorem applied to the original submersion gives
$b_2(X_0)=b_2(X_t)$, and
Proposition~\ref{prop:relative-model} identifies $X_t$ with $Y_t$
for $t\ne0$.
\end{proof}

\subsection{Nearby perverse sheaves and Jordan--H\"older filtrations}
\label{sec:nearby}

Let $g\colon Y\to\Delta$ be flat and projective, with $Y$ normal
of pure dimension $n+1$, $Y_0$ normal and irreducible, and $Y_t$
smooth for $t\ne0$. Define the ordinary rational perverse sheaf
\begin{equation}\label{eq:J-general}
 \mathcal J:={}^p\!\psi_g\IC_Y
 =R\psi_g\IC_Y[-1]\in\Perv_{\Q}(Y_0).
\end{equation}
The shifted nearby-cycle functor is exact for the perverse
$t$-structure \cite[Theorem~5.5.3(2)]{deCataldoMigliorini2009}.
Since $Y$ is smooth over $\Delta^*$ and nearby cycles depend only
on the restriction there, proper comparison gives
\begin{equation}\label{eq:J-rational}
 \mathcal J\simeq R\psi_g\Q_Y[n],\qquad
 \mathbb H^{2-n}(Y_0,\mathcal J)
 \simeq H^2(Y_t,\Q).
\end{equation}

\begin{lemma}
\label{lem:nearby-factors}
The perverse sheaf $\mathcal J$ has finite length. Exactly one
simple composition factor has full support: it is $\IC_{Y_0}$,
and it occurs with multiplicity one. Every other composition
factor is supported in $\Sing Y_0$ and has support dimension at
most $n-2$.
\end{lemma}

\begin{proof}
The category of constructible perverse sheaves has finite length;
its simple objects are intermediate extensions of irreducible
local systems on smooth dense open subsets of their irreducible
supports \cite[Section~2.3]{deCataldoMigliorini2009}.
Put $U=Y_{0\mathrm{reg}}$. Flatness and the smoothness of the
central fiber at each point of $U$ imply that $g$ is smooth near
$U$. Consequently
\[
 \mathcal J|_U\simeq\Q_U[n].
\]
The open set $U$ is connected because $Y_0$ is normal and
irreducible. Restriction to $U$ is exact for perverse sheaves,
and $\Q_U[n]$ is simple. Hence the composition factors with
nonzero restriction to $U$ consist of one copy of $\IC_{Y_0}$.
Every other factor is supported in $\Sing Y_0$, which has
dimension at most $n-2$ by normality.
\end{proof}

\begin{proposition}\label{prop:only-codim2}
Choose a Jordan--H\"older filtration of $\mathcal J$. For each
irreducible codimension-two component of $\Sing Y_0$, choose a
connected dense smooth open stratum $S_\alpha$, small enough in
the Zariski topology that every composition factor with this
support restricts to a local system shifted by $[n-2]$.
Denote these coefficient local systems by $M_{\alpha,j}$, with
each occurrence retained, and put
$M_\alpha=\bigoplus_jM_{\alpha,j}$; the direct sum is zero if
there is no such factor. Then
\begin{equation}\label{eq:strict-support-general}
 b_2(Y_t)
 \le\Ib_2(Y_0)+\sum_{\alpha,j}h^0(S_\alpha,M_{\alpha,j})
 =\Ib_2(Y_0)+\sum_\alpha h^0(S_\alpha,M_\alpha).
\end{equation}
\end{proposition}

\begin{proof}
Choose a composition series
\[
 0=\mathcal J_0\subset\mathcal J_1\subset\cdots
 \subset\mathcal J_r=\mathcal J,
 \qquad
 \mathcal S_j=\mathcal J_j/\mathcal J_{j-1}.
\]
Each filtration short exact sequence gives a distinguished
triangle and hence an exact sequence
\[
 \mathbb H^k(Y_0,\mathcal J_{j-1})
 \longrightarrow\mathbb H^k(Y_0,\mathcal J_j)
 \longrightarrow\mathbb H^k(Y_0,\mathcal S_j).
\]
The middle term has dimension at most the sum of the dimensions
of the outer terms. Induction therefore gives, for every $k$,
\begin{equation}\label{eq:jh-hypercohomology}
 \dim_{\Q}\mathbb H^k(Y_0,\mathcal J)
 \le\sum_{j=1}^r
 \dim_{\Q}\mathbb H^k(Y_0,\mathcal S_j).
\end{equation}

By Lemma~\ref{lem:nearby-factors}, the unique full-support
factor contributes $\Ib_2(Y_0)$ in degree $2-n$.
Let $\IC_{\overline S}(M)$ be another factor, where $S$ is smooth
of dimension $s\le n-2$. Its ordinary cohomology sheaves vanish
below degree $-s$. If $s<n-2$, it therefore has no hypercohomology
in degree $2-n$. If $s=n-2$, its contribution in that degree is
\[
 \mathbb H^{2-n}
   (\overline S,\IC_{\overline S}(M))
 \simeq H^0(S,M)
\]
by \eqref{eq:IC-lowest}. In particular, all proper supports that
can contribute are precisely among the codimension-two
components used in the statement. Applying
\eqref{eq:jh-hypercohomology} in degree $2-n$, and using
\eqref{eq:J-rational}, gives \eqref{eq:strict-support-general}.
\end{proof}

\section{Codimension-two ADE systems and Weyl-twisted monodromy}\label{sec:ADE}

Throughout this section, $g\colon Y\to\Delta$ satisfies the hypotheses
of Theorem~\ref{thm:nearby-general}. By a stratum $S_\alpha$ we always mean
a connected dense Zariski open subset of an irreducible
codimension-two component of $\operatorname{Sing}Y_0$.
Replacing $S_\alpha$ by a smaller stratum means deleting a proper
closed algebraic subset. Small simply connected analytic charts
will be used to construct markings, but will never replace
$S_\alpha$ in any monodromy group or invariant space.
In particular, all global sections and monodromy orbits below are
computed on the entire chosen dense open stratum.

If $S'\subset S_\alpha$ is a further dense Zariski open restriction,
the map $\pi_1(S')\to\pi_1(S_\alpha)$ is surjective: loops can be
moved off a proper complex algebraic subset. Thus the monodromy
image and the invariant space of an already defined local system are
unchanged by such a restriction.

\begin{lemma}\label{lem:ADE-transverse-nearby}
The strata $S_\alpha$ can be chosen so that their transverse Du Val
type is constant and the following holds. Put $d=n-2$. At every
$p\in S_\alpha$, choose sufficiently general local holomorphic
functions $u=(u_1,\ldots,u_d)$ whose restriction to $S_\alpha$ is a
coordinate system near $p$. There is a neighborhood of $p$ on which
\[
 (u,g)\colon Y\longrightarrow B\times\Delta
\]
is a flat family of surface germs, with $B\subset\mathbb C^d$ a
small polydisc. The central transverse surface
\[
 T_p=Y_0\cap\{u=u(p)\}
\]
has a Du Val singularity at $p$, and the transverse fibers with
$t\ne0$ are smooth after making the neighborhood and disc smaller.
If $F_p$ is a transverse Milnor fiber, restriction induces
\begin{equation}\label{eq:ADE-transverse-stalk}
 \mathcal H^{2-n}(\mathcal J)_p
 \simeq H^2(F_p,\mathbb Q).
\end{equation}
These identifications commute with continuation along
$S_\alpha$ and with the nearby-cycle monodromy in $t$.
Locally along $S_\alpha$, the family of central surface germs is
analytically trivial.
\end{lemma}

\begin{proof}
Choose a complex analytic Whitney stratification near $Y_0$,
compatible with $g^{-1}(0)$ and satisfying Thom's $a_g$ condition.
The existence theorem applies to a holomorphic map to a curve;
see \cite{Hironaka1977Stratification} and the discussion of the
analytic-domain case in \cite[Introduction]{Massey2007Thom}.
It can be refined to be compatible with the singular locus and with
a fixed resolution of $Y_0$. Its restriction to the compact
projective variety $Y_0$ is finite, and the closures of the strata are
closed analytic subsets of $Y_0$, hence algebraic. Thus the
top-dimensional stratum in each relevant singular component is a
dense Zariski open subset, and we delete the remaining proper closed
subsets whenever necessary below. This operation retains the
global stratum; it is distinct from passing to a coordinate chart.

For a fixed resolution of $Y_0$, general transverse hyperplane
sections have normal canonical surface germs by Bertini for
discrepancies and adjunction
\cite[Lemma~5.17(1), Corollary~5.18]{KollarMori1998}.
Indeed, the hyperplanes can be chosen not to contain any of the
finitely many exceptional centers, so their total transforms have
no exceptional divisorial parts. On their smooth complete
intersection in the resolution, adjunction restricts the
nonnegative discrepancies of $Y_0$. A normal canonical surface
singularity is Du Val \cite[Theorem~4.20]{KollarMori1998}.
The ADE type is constant after the indicated dense open restriction.

We next verify transversality to the nearby fibers. This does not
follow merely from their smoothness. Extend $u$ to an ambient
smooth embedding near $p$. Thom's $a_g$ condition says that if
$y_j\in Y_{t_j}$, $t_j\ne0$, tends to a point $p'\in S_\alpha$,
and the planes $T_{y_j}Y_{t_j}$ tend to a plane $\tau$, then
\[
 T_{p'}S_\alpha\subset\tau.
\]
Suppose that $du|_{T_yY_t}$ failed to be surjective at points
arbitrarily close to $p$. Along a convergent sequence, choose unit
covectors $\lambda_j\in(\mathbb C^d)^*$ annihilating
$du(T_{y_j}Y_{t_j})$. A subsequence has limits $\lambda\ne0$ and
$\tau$. Then $\lambda\circ du$ annihilates $\tau$, and therefore
$T_pS_\alpha$. This contradicts the invertibility of
$du|_{T_pS_\alpha}$. After shrinking the coordinate neighborhood,
$u|_{Y_t}$ is consequently a submersion for every sufficiently
small $t\ne0$, so all its transverse surface fibers there are
smooth. The same argument applies at every point of the chart.

For completeness, flatness of the surface family follows from
local commutative algebra. The canonical central fiber is
Cohen--Macaulay. Since $g$ is flat and its base is smooth of
dimension one, $\mathcal O_{Y,p}$ is Cohen--Macaulay as well.
The sequence
\[
 t,u_1-u_1(p),\ldots,u_d-u_d(p)
\]
is a partial system of parameters: its quotient is the
two-dimensional transverse surface ring. It is therefore a regular
sequence. The local flatness criterion over the regular local ring
of $B\times\Delta$ gives flatness
of $(u,g)$ near $p$.

Take a small
normal Milnor ball for $T_p$ in the slice $u=u(p)$, with boundary
transverse to its smooth locus. Keep its radius fixed and restrict
$u$ to a smaller polydisc. At the boundary the central transverse
fibers are smooth; flatness then makes $(u,g)$ smooth there, and
the boundary transversality persists for small $u-u(p)$ and $t$.
In the interior, the submersion assertion just proved supplies
transversality for $t\ne0$. Thus, for each sufficiently small
$t\ne0$, the part of $Y_t$ in this closed normal tube is a proper
submersion, including along its boundary, over the $u$-polydisc.
Ehresmann's theorem for manifolds with boundary makes it a fiber
bundle over that contractible polydisc, and the inclusion of the fiber
over $u(p)$ is a cohomology equivalence. These normal tubes give
the usual cofinal Milnor representatives defining the stalk of
$R\psi_g\mathbb Q_Y$. Consequently
\[
 (R^k\psi_g\mathbb Q_Y)_p\simeq H^k(F_p,\mathbb Q).
\]
With $\mathcal J=R\psi_g\mathbb Q_Y[n]$, this gives
\eqref{eq:ADE-transverse-stalk}. The maps used here are restriction
maps, so changing tubes or coordinates through transverse choices
preserves the identifications and their continuation. Equivalently,
this is non-characteristic restriction for nearby cycles:
Thom's condition places their microsupport, above the stratum,
in its conormal bundle
\cite[Theorem~3.7 and Corollary~3.8]{Massey2007Thom}, and
invertibility of $du|_{TS_\alpha}$ makes the normal slice
non-characteristic. The construction can be performed after
pullback to the universal cover of $\Delta^*$, which also proves
compatibility with the $t$-monodromy.

Finally, a flat deformation of a Du Val surface germ is obtained
from its semiuniversal deformation, whose base is
$\mathfrak h/W$ \cite[Sections~8.7 and~8.9]{Slodowy1980}.
Apply this to $(u,g)$ near $(u(p),0)$, obtaining a classifying map
\[
 \varphi\colon B\times\Delta\longrightarrow\mathfrak h/W.
\]
In the semiuniversal deformation of a simple ADE singularity, the
stratum retaining the same full ADE type is the origin: away from the
origin, the discriminant strata correspond to proper root subsystems. Since all central transverse germs in $B$ have
the original type, $\varphi(B\times\{0\})=\{0\}$.
The pullback description therefore identifies the central family
of surface germs with the product of $B$ and the fixed Du Val
germ.
\end{proof}

For a local example in dimension three, consider
\[
 V=\{xy-z^2=0\}\subset\C^4_{x,y,z,s}.
\]
Its singular locus is the curve $S=\{x=y=z=0\}$. At
$p=(0,0,0,s_0)$, the section $s=s_0$ is the surface
$\{xy-z^2=0\}\subset\C^3$, with an $A_1$ singularity. Here $s$ is a
coordinate along the singular stratum. The minimal resolution of the
transverse surface replaces its singular point by one $(-2)$-curve;
for a general Du Val type, the exceptional $(-2)$-curves instead form
the corresponding ADE configuration, which supplies the root lattice
used in this section.

Let $\Phi_\alpha$ be this transverse root system, let
$\Pi_\alpha$ be its simple roots, and let $V_\alpha$ be its real
root space. The link of a Du Val surface singularity is a
rational homology sphere. By the local analytic product of
Lemma~\ref{lem:ADE-transverse-nearby},
\begin{equation}\label{eq:IC-constant}
 \IC_{Y_0}\simeq\mathbb Q_{Y_0}[n]
\end{equation}
near the chosen dense strata. Define on the entire stratum
\begin{equation}\label{eq:Milnor-system}
 \mathcal L_\alpha:=\mathcal H^{2-n}(\mathcal J)|_{S_\alpha}.
\end{equation}
Lemma~\ref{lem:ADE-transverse-nearby} identifies its stalks with
the second cohomology of the transverse Milnor fiber. The Milnor
homology is the ADE root lattice with its negative definite
intersection pairing. Its dual is Milnor cohomology, and the
pairing identifies these spaces rationally and equivariantly.
Continuation on Milnor homology preserves the integral root
lattice and the pairing, hence preserves the roots (its vectors
of square $-2$). Thus the monodromy image is finite:
\[
 H_\alpha\subset\operatorname{Aut}(\Phi_\alpha)
 =W_\alpha\rtimes\operatorname{Aut}(\Pi_\alpha).
\]
Here the monodromy defining $H_\alpha$ runs along $S_\alpha$;
it is distinct from the monodromy around $t=0$.

\begin{lemma}\label{lem:milnor-filtration}
For every $\alpha$, the coefficient local systems entering
$M_\alpha$ in Proposition~\ref{prop:only-codim2} are the successive
nonzero quotients of a filtration of $\mathcal L_\alpha$.
In particular,
\begin{equation}\label{eq:milnor-filtration-invariants}
 h^0(S_\alpha,M_\alpha)
 =h^0(S_\alpha,\mathcal L_\alpha)
 =\dim_{\Q} V_{\alpha,\Q}^{H_\alpha}.
\end{equation}
\end{lemma}

\begin{proof}
Fix the Jordan--H\"older filtration used in
Proposition~\ref{prop:only-codim2}, and write it as
\[
 0=\mathcal J_0\subset\mathcal J_1\subset\cdots
 \subset\mathcal J_r=\mathcal J,
 \qquad
 \mathcal S_j=\mathcal J_j/\mathcal J_{j-1}.
\]
Shrink $S_\alpha$ by deleting its intersections with the other
proper supports and their boundaries. This is a dense Zariski
open restriction and does not change the invariant spaces of
the already defined local systems. By
Lemma~\ref{lem:ADE-transverse-nearby} and
\eqref{eq:IC-constant}, the restriction of the unique
full-support factor to $S_\alpha$ is $\Q_{S_\alpha}[n]$.
The remaining nonzero factor restrictions are precisely
\[
 \mathcal S_j|_{S_\alpha}
 \simeq M_{\alpha,j}[n-2],
\]
where $M_{\alpha,j}$ are the summands, with multiplicities,
defining $M_\alpha$.

Thus each $\mathcal S_j|_{S_\alpha}$ has ordinary cohomology
only in degrees $-n$ and $2-n$. In particular,
\[
 \mathcal H^{1-n}(\mathcal S_j)|_{S_\alpha}=0,
 \qquad
 \mathcal H^{3-n}(\mathcal S_j)|_{S_\alpha}=0.
\]
The distinguished triangles associated with the filtration show
inductively that the same two vanishings hold for every
$\mathcal J_j|_{S_\alpha}$. Applying ordinary sheaf cohomology
in degree $2-n$ to these triangles consequently gives short
exact sequences of local systems
\[
 0\longrightarrow
 \mathcal H^{2-n}(\mathcal J_{j-1})|_{S_\alpha}
 \longrightarrow
 \mathcal H^{2-n}(\mathcal J_j)|_{S_\alpha}
 \longrightarrow
 \mathcal H^{2-n}(\mathcal S_j)|_{S_\alpha}
 \longrightarrow0.
\]
The last term is $M_{\alpha,j}$ for an $S_\alpha$-supported
factor and is zero otherwise. Since
\[
 \mathcal H^{2-n}(\mathcal J)|_{S_\alpha}
 =\mathcal L_\alpha,
\]
these sequences provide the asserted filtration.

The monodromy group $H_\alpha$ of $\mathcal L_\alpha$ is finite.
Every subobject and quotient in this filtration is therefore a
representation of the same finite group. Over $\Q$, taking
$H_\alpha$-invariants is exact, by averaging. Summing invariant
dimensions along the short exact sequences gives
\[
 h^0(S_\alpha,\mathcal L_\alpha)
 =\sum_j h^0(S_\alpha,M_{\alpha,j})
 =h^0(S_\alpha,M_\alpha).
\]
Finally, the ADE intersection pairing identifies the Milnor
cohomology representation with $V_{\alpha,\Q}$ equivariantly.
This proves \eqref{eq:milnor-filtration-invariants}.
\end{proof}

The minimal resolution of the transverse Du Val singularity carries the simple roots as effective $(-2)$-curves. Continuing these along $S_\alpha$ gives a Dynkin-diagram monodromy group
\[
 \Gamma_\alpha\subset\Aut(\Pi_\alpha).
\]
The Milnor monodromy $H_\alpha$ may contain Picard--Lefschetz reflections and therefore need not preserve the chosen chamber. What is true is the following comparison modulo the Weyl group.

\begin{lemma}\label{lem:ADE-outer-naturality}
Let $\sigma\colon(T,0)\to(T',0)$ be an analytic isomorphism of
Du Val surface germs of the same type. It lifts uniquely to an
isomorphism of their minimal resolutions and hence induces a
permutation $\sigma_\Pi$ of the effective simple roots. Compare
the root lattices of smooth deformation fibers by an isomorphism
of deformations extending $\sigma$ and by continuation in the
smooth deformation loci. With any simultaneous-resolution
markings, the resulting root-lattice isometry has image
$\sigma_\Pi$ in $\operatorname{Aut}(\Phi)/W$. This image is
independent of the choices of markings and continuation paths.
\end{lemma}

\begin{proof}
Fix the effective simple-root bases on the two central minimal
resolutions; all central lattice identifications in this proof
refer to these bases.
Write $\mathscr T\to D=\mathfrak h/W$ for a semiuniversal
deformation of $(T,0)$ and $D^{\rm reg}$ for its smooth-fiber
locus. The simultaneous-resolution theorem supplies a smooth
family resolving its pullback to $\mathfrak h$, whose central
fiber is the minimal resolution of $T$. Over $D^{\rm reg}$,
$\mathfrak h^{\rm reg}\to D^{\rm reg}$ is a $W$-cover, and the
monodromy of the vanishing root lattice is exactly $W$.
These statements are the marked simultaneous-resolution and
monodromy theorems
\cite[Sections~8.7 and~8.9]{Slodowy1980},
\cite[Section~2 and Theorem~2.11(1),(2),(4)]{ShepherdBarron2021}; the
monodromy description is also recalled in
\cite[Remark~1.3(ii)]{KerrLazaSaito2022}.

The marked-resolution functors and their change-of-marking
action are described in Section~2 of the latter reference:
the full marking group is $\operatorname{Aut}(\Phi)$, and
the chamber stabilizer is $\operatorname{Aut}(\Pi)$.
Their quotient by the Weyl subgroup records the action on
simple roots after adjustment to the chosen chamber.
Consequently, the comparisons of a smooth-fiber root lattice with
the central resolution lattice furnished by lifts to the
simultaneous-resolution base form a $W$-torsor. On a sufficiently
small simply connected lift, this comparison can be realized by
parallel transport in the smooth family of resolved Milnor
representatives. The central resolved representative retracts
onto its exceptional curves, so its second homology is their
root lattice. Changing a lift, or making a loop in
$D^{\rm reg}$, changes the comparison by a Weyl element.

Now consider an isomorphism of deformation families extending
$\sigma$. Pull a chosen simultaneous resolution of the target
family back through this isomorphism. This is a simultaneous
resolution of the source family, and its central identification with
the target resolution is the unique lift of $\sigma$ to minimal
resolutions. This lift takes every irreducible exceptional curve
to the corresponding irreducible exceptional curve, so its action
on the central root basis is $\sigma_\Pi$. Parallel transport in
the pulled-back smooth resolved family shows that, for these
compatible markings, the action on the smooth-fiber lattice is
also $\sigma_\Pi$. Replacing the pulled-back marking by any
other marking multiplies this isometry by a Weyl element, and
additional continuation in the smooth deformation locus does
the same. This proves the claimed identity modulo $W$.

\end{proof}

\begin{lemma}\label{lem:mod-Weyl}
Under the natural identification
\[
 \Aut(\Phi_\alpha)/W_\alpha\simeq\Aut(\Pi_\alpha),
\]
the image of $H_\alpha$ is $\Gamma_\alpha$, up to conjugacy.
\end{lemma}

\begin{proof}
Cover the entire dense stratum $S_\alpha$ by small simply
connected coordinate charts. Lemma~\ref{lem:ADE-transverse-nearby}
provides, on each chart, a flat family of transverse surface germs
and an analytic identification of its central family with the
fixed ADE germ times the chart. The minimal resolutions of those
central germs identify their effective exceptional curves, and
their transitions define the diagram local system with monodromy
$\Gamma_\alpha$.

On each chart, use a classifying map to $\mathfrak h/W$ and choose
lifts to the simultaneous-resolution base over simply connected
sectors of $\Delta^*$. These give markings of the transverse
Milnor local system; different sectorwise lifts change the markings
by Weyl elements.

We check the transitions rather than identifying the two full
representations. At a point of an overlap, the two central
identifications differ by an analytic isomorphism $\sigma$ of
the fixed Du Val germ. Its action on the exceptional curves is
the transition of the diagram local system. When the transverse
projections agree, Lemma~\ref{lem:ADE-outer-naturality} shows
that the corresponding Milnor transition has exactly this image
modulo $W_\alpha$. If the projections differ, call them $u$ and
$u'$. Make a holomorphic change of coordinates in the target of
$u'$ so that $u'|_{S_\alpha}=u|_{S_\alpha}$ on a smaller overlap,
and set
\[
 u_\lambda=(1-\lambda)u+\lambda u',\qquad 0\le\lambda\le1.
\]
Every $du_\lambda|_{TS_\alpha}$ equals $du|_{TS_\alpha}$
and is invertible. In a central analytic product chart, the
implicit function theorem identifies each central transverse
germ for $u_\lambda$ with the same Du Val germ. The tangent-space
argument and the normal-tube construction of
Lemma~\ref{lem:ADE-transverse-nearby} apply uniformly in the
compact parameter interval, after reducing the neighborhood.
Thus this path of slices supplies continuation among smooth
deformations of the fixed ADE germ, and its ambiguity is a loop in
the smooth semiuniversal deformation locus, whose action lies in
$W_\alpha$. The same conclusion modulo $W_\alpha$ therefore
holds for different projections and different Milnor
representatives.

Compose these transitions along a loop in $S_\alpha$. Weyl
elements disappear in the quotient, leaving precisely the
transitions of the diagram local system. Thus the projection of
the Milnor monodromy representation to
$\operatorname{Aut}(\Phi_\alpha)/W_\alpha$ is the diagram
monodromy representation, up to the initial marking, and taking
images gives the assertion.
\end{proof}

\begin{lemma}\label{lem:root-inequality}
Let $\Phi$ be a finite reduced root system with real root space $V$, Weyl group $W$, and simple roots $\Pi$. Let $H\subset\Aut(\Phi)$, and let $\Gamma\subset\Aut(\Pi)$ be its image modulo $W$. Then
\begin{equation}\label{eq:root-inequality}
 \dim V^H\le\dim V^\Gamma=|\Pi/\Gamma|.
\end{equation}
\end{lemma}

\begin{proof}
Let $F=V^H$. If $F=0$ there is nothing to prove. Choose a general $x\in F$ so that a root vanishes on $x$ exactly when it vanishes identically on $F$. After Weyl conjugation, assume $x$ belongs to the closure of the fixed fundamental chamber $C$.

Let $W_x$ be the stabilizer of $x$ in $W$. It is generated by the reflections in roots vanishing at $x$, hence it fixes $F$ pointwise. For $h\in H$, the chamber $hC$ also has $x$ in its closure. The group $W_x$ acts transitively on chambers whose closures contain $x$, so there is $w_h\in W_x$ with
\[
 w_hhC=C.
\]
Then $\gamma_h:=w_hh$ stabilizes $C$, hence belongs to $\Aut(\Pi)$ and represents the image of $h$ modulo $W$. For $v\in F$,
\[
 \gamma_hv=w_hhv=w_hv=v.
\]
Thus $F\subset V^\Gamma$. Finally, $\Gamma$ permutes the basis $\Pi$, and $V^\Gamma$ is spanned by the sums over $\Gamma$-orbits. Hence $\dim V^\Gamma=|\Pi/\Gamma|$.
\end{proof}

\begin{definition}
Let $c_{2,\alpha}$ be the number of discrepancy-zero divisorial valuations of $Y_0$ whose center is the closure of $S_\alpha$, and put
\[
 c_2(Y_0):=\sum_\alpha c_{2,\alpha}.
\]
\end{definition}

\begin{lemma}\label{lem:c2-orbits}
For every $\alpha$,
\begin{equation}\label{eq:c2-orbits}
 c_{2,\alpha}=|\Pi_\alpha/\Gamma_\alpha|.
\end{equation}
If $Y_0$ is the canonical model of the smooth central fiber $X_0$, and $c$ denotes the number of crepant prime divisors on $X_0$ contracted by the canonical map but not fixed, then
\begin{equation}\label{eq:c2-le-c}
 c_2(Y_0)\le c.
\end{equation}
\end{lemma}

\begin{proof}
On a transverse Du Val surface, the exceptional curves of the minimal resolution are exactly the discrepancy-zero divisorial valuations, while further point blowups have positive discrepancy. As these curves move along $S_\alpha$, the graph monodromy groups them precisely into $\Gamma_\alpha$-orbits. An orbit closes up to one divisorial valuation over the generic point of $S_\alpha$, which proves \eqref{eq:c2-orbits}.

Let $v$ be such a discrepancy-zero valuation and choose a common smooth projective resolution carrying $v$ and principalizing the canonical base ideal. From \eqref{eq:fixed-part},
\[
 0=a(v;Y_0)=a(v;X_0)+\ord_vF.
\]
Both terms on the right are nonnegative. Since $X_0$ is smooth, an exceptional divisor over $X_0$ has positive discrepancy. Therefore $v$ must be the ordinary valuation of a prime divisor on $X_0$, and $\ord_vF=0$. It is contracted by the canonical map and is crepant, but is not fixed. Distinct valuations give distinct prime divisors, which proves \eqref{eq:c2-le-c}.
\end{proof}

\begin{theorem}\label{thm:nearby-general}
Let $g\colon Y\to\Delta$ be flat and projective, with $Y$ normal
of pure dimension $n+1$, $Y_0$ normal and irreducible with
canonical singularities, and $Y_t$ smooth for $t\ne0$. Then
\begin{equation}\label{eq:nearby-general}
 b_2(Y_t)\le\Ib_2(Y_0)+c_2(Y_0).
\end{equation}
\end{theorem}

\begin{proof}
Proposition~\ref{prop:only-codim2} and
Lemma~\ref{lem:milnor-filtration} give
\[
 b_2(Y_t)
 \le\Ib_2(Y_0)
   +\sum_\alpha\dim_{\Q}V_{\alpha,\Q}^{H_\alpha}.
\]
By Lemmas~\ref{lem:mod-Weyl}, \ref{lem:root-inequality}, and
\ref{lem:c2-orbits},
\[
 \sum_\alpha\dim_{\Q}V_{\alpha,\Q}^{H_\alpha}
 \le\sum_\alpha|\Pi_\alpha/\Gamma_\alpha|
 =c_2(Y_0).
\]
This proves \eqref{eq:nearby-general}.
\end{proof}

We now combine the two independent comparisons.
\begin{theorem}\label{thm:main-divisorial}
Let
$\pi\colon\mathcal X\to\Delta$ be a proper holomorphic submersion
with connected compact fibers of dimension $n\ge1$ such that
$K_{X_t}$ is ample for $t\ne0$. Then $K_{X_0}$ is big, $X_0$ is Moishezon, and
$\operatorname{codim}_{X_0}\mathbf B(K_{X_0})\ge2$.
There is a flat projective relative canonical model
$g\colon Y\to\Delta$ with normal total space, $Y_t\simeq X_t$ for
$t\ne0$, and $Y_0=\Proj R(X_0,K_{X_0})$.
Let $c$ be the number of prime divisors on $X_0$ contracted crepantly
by the canonical map, and let $c_2(Y_0)$ be the number of
discrepancy-zero divisorial valuations with codimension-two center
on $Y_0$. Then
\[
 c=c_2(Y_0)=b_2(X_t)-\Ib_2(Y_0)\qquad(t\ne0).
\]
\end{theorem}

\begin{proof}
Proposition~\ref{prop:bigness} gives the bigness of $K_{X_0}$ and the Moishezon property of $X_0$, and Proposition~\ref{prop:relative-model} constructs the flat projective canonical model $g\colon Y\to\D$. Lemma~\ref{lem:model-normality} gives normality of $Y$ and
$b_2(X_0)=b_2(Y_t)$. Thus the nearby-cycle comparison of
Theorem~\ref{thm:nearby-general} applies.

Let $r$ be the number of stable fixed prime divisors of $K_{X_0}$, let $c$ be the number of the other crepant prime divisors on $X_0$ contracted by the canonical map, and let $c_2=c_2(Y_0)$. Theorem~\ref{thm:ih-count-general}, Theorem~\ref{thm:nearby-general}, and Lemma~\ref{lem:c2-orbits} give
\begin{align*}
 r+c
 &=b_2(X_0)-\Ib_2(Y_0)\\
 &=b_2(Y_t)-\Ib_2(Y_0)\\
 &\le c_2\\
 &\le c.
\end{align*}
Since $r\ge0$, it follows that
\[
 r=0,
 \qquad
 c=c_2=b_2(Y_t)-\Ib_2(Y_0).
\]
Hence the stable base locus has no divisorial component, which completes the proof.
\end{proof}

\section{The central fiber is K\"ahler in dimension three}
\label{sec:base-elimination}

We now assume $n=3$. Choose the divisible integer $m$ of Proposition~\ref{prop:relative-model} and put $B=\Bs|mK_{X_0}|=\mathbf B(K_{X_0})$. Theorem~\ref{thm:main-divisorial} gives
\begin{equation}\label{eq:two-consequences}
 \dim B\le1,\qquad c=c_2(Y_0).
\end{equation}
On a common smooth projective resolution
$p\colon W\to X_0$, $q\colon W\to Y_0$ principalizing the base ideal,
Lemma~\ref{lem:fixed-part} gives
\begin{equation}\label{eq:fixed}
 p^*K_{X_0}=q^*K_{Y_0}+F,\qquad F\ge0,\qquad
 \mathfrak b_m\mathcal O_W=\mathcal O_W(-mF),
\end{equation}
with $F$ $q$-exceptional.

\begin{lemma}\label{lem:target-open}
There exists a Zariski-open subset \(U\subset Y_{0\mathrm{reg}}\)
such that the inverse canonical map identifies \(U\)
biholomorphically with an open subset of \(X_0\setminus B\), and
\[
  D:=Y_0\setminus U,\qquad \dim D\le1.
\]
\end{lemma}

\begin{proof}
Consider the common resolution in \eqref{eq:fixed}.
Every \(p\)-exceptional prime divisor \(P\) has positive discrepancy
over the smooth manifold \(X_0\), so the discrepancy identity gives
\[
  a(P;Y_0)=a(P;X_0)+\operatorname{ord}_P(F)>0.
\]
Hence \(P\) is also \(q\)-exceptional.

Remove from \(Y_0\) its singular locus and the images of all
\(q\)-exceptional divisors. Each removed subset has codimension at
least two. Since a modification onto a smooth space has purely
divisorial exceptional locus, the resulting open set \(U\) is an
isomorphism locus for both maps on the common resolution.
Moreover, \(F\) vanishes above \(U\), so its image in \(X_0\)
avoids \(B\). This proves the assertion.
\end{proof}

Let \(\mathcal I\subset\mathcal O_{\mathcal X}\) denote the relative
evaluation ideal. For the evaluation morphism
\[
  \mathrm{ev}:\pi^*\pi_*\omega_{\mathcal X/\Delta}^{\otimes m}
    \longrightarrow\omega_{\mathcal X/\Delta}^{\otimes m},
\]
it is defined by
\[
 \operatorname{im}(\mathrm{ev})
   =\mathcal I\otimes\omega_{\mathcal X/\Delta}^{\otimes m}.
\]
Its zero set is \(B\subset X_0\), because of base change and
generation on the punctured disc.
Let
\[
  a:\Gamma\longrightarrow\mathcal X,\qquad
  b:\Gamma\longrightarrow\mathcal Y
\]
be the normalization of the closure of the graph of the relative
canonical map. Indeed, the map is given by the degree-$m$ evaluation sections into
the projective bundle $\mathbb P(E_m^*)$, and its image lies in
$\mathcal Y$, so the usual graph construction identifies
$a$ with the normalized blow-up of $\mathcal I$. The map
\[
  (a,b):\Gamma\longrightarrow
       \mathcal X\times_\Delta\mathcal Y
\]
is finite onto the closed graph.

\begin{lemma}\label{lem:connected}
For the sets \(B\) and \(D\) above,
\begin{equation}\label{eq:bubble-target}
  b\bigl(a^{-1}(B)\bigr)\subset D.
\end{equation}
\end{lemma}

\begin{proof}
Fix \(y\in U\), and let \(x\in X_0\setminus B\) be the corresponding
point. The relative canonical map is holomorphic near \(x\).
The space \(\mathcal Y\) is smooth near \(y\), and \(g\) is a
submersion there: this follows from flatness and the smoothness of
the central fiber at \(y\). On the central tangent spaces the
canonical map is an isomorphism, and on the normal direction to
the central fiber it induces the identity on the base.
The holomorphic inverse function theorem therefore shows that
the relative canonical map is locally biholomorphic near \(x\).

It follows that \(b^{-1}(y)\) contains an isolated point.
The proper bimeromorphic map \(b\) has connected fibers because
\(\mathcal Y\) is normal. Thus \(b^{-1}(y)\) consists of that
single point, whose \(a\)-image is \(x\notin B\).
No point of \(a^{-1}(B)\) can therefore map to \(y\).
\end{proof}

\begin{proposition}\label{prop:semiample}
The set \(B\) is empty. In particular, the canonical map is a
proper birational morphism
\begin{equation}\label{eq:canonical-morphism}
  f:X_0\longrightarrow Y_0,\qquad K_{X_0}=f^*K_{Y_0}.
\end{equation}
\end{proposition}

\begin{proof}
Suppose \(B\ne\varnothing\). The ideal \(\mathcal I\) is nontrivial,
and its zero locus has codimension at least three in the smooth
fourfold \(\mathcal X\). Its pullback to the normalized blow-up is an
invertible ideal, and it is proper above any point of $B$, because the
original generators all vanish there. It therefore defines a nonzero
effective Cartier divisor.
Choose a prime component \(E\) of that divisor. Then
\[
  \dim E=3,\qquad a(E)\subset B.
\]
By Lemma~\ref{lem:connected}, \(b(E)\subset D\).
Since the normalized graph is finite over the closed graph,
\[
  3=\dim E
    =\dim(a,b)(E)
    \le\dim(B\times D)
    \le1+1=2,
\]
a contradiction. Hence \(B=\varnothing\).

The complete base-point-free system \(|mK_{X_0}|\) now gives the
canonical morphism \(f\), with
\[
  \mathcal O_{X_0}(mK_{X_0})\simeq f^*\mathcal O_{Y_0}(mK_{Y_0}).
\]
Equivalently, \eqref{eq:fixed} has \(F=0\).
Thus \eqref{eq:canonical-morphism} holds for compatible canonical
divisors.
\end{proof}

\subsection{The exceptional fibers are curves}\label{sec:exceptional-fibers}

\begin{proposition}\label{prop:fiber-dimension}
Every exceptional prime divisor of \(f\) maps onto a curve.
In particular,
\begin{equation}\label{eq:fiber-dimension}
  \dim f^{-1}(y)\le1\qquad\text{for every }y\in Y_0.
\end{equation}
\end{proposition}

\begin{proof}
By \eqref{eq:canonical-morphism}, all \(f\)-exceptional prime
divisors are crepant. Conversely, a crepant divisorial valuation
\(v\) over \(Y_0\) satisfies
\[
  a(v;Y_0)=a(v;X_0)
\]
because \(f\) is crepant. If its center on \(X_0\) had codimension
at least two, smoothness of \(X_0\) would give \(a(v;X_0)>0\).
Hence every such valuation is represented by a prime divisor on
\(X_0\).

In particular, the crepant valuations with codimension-two center
on \(Y_0\) form a subset of the exceptional prime divisors on \(X_0\).
The equality \(c=c_2(Y_0)\) shows that this subset
already accounts for all exceptional prime divisors, so there is no
exceptional prime divisor mapping to a point.
A two-dimensional component of a fiber of a birational morphism
from a threefold would be precisely such a prime divisor.
This proves \eqref{eq:fiber-dimension}.
\end{proof}

\begin{lemma}\label{lem:rational-fibers}
Every irreducible curve contracted by \(f\) is a smooth rational
curve.
\end{lemma}

\begin{proof}
Canonical singularities are rational \cite[Theorem~5.22]{KollarMori1998}, so
\[
  R^1f_*\mathcal O_{X_0}=0.
\]
This vanishing does not require prior projectivity of \(f\):
choose a smooth projective modification $h\colon W\to X_0$ and put
$q=f\circ h$. The map $q$ is a projective resolution of $Y_0$,
so $R^1q_*\mathcal O_W=0$. Since $h_*\mathcal O_W=\mathcal O_{X_0}$,
the Leray spectral sequence gives an injection
\[
 R^1f_*\mathcal O_{X_0}\hookrightarrow R^1q_*\mathcal O_W=0.
\]

Fix \(y\in Y_0\), and let \(F_k\) be the infinitesimal neighborhood of
the scheme-theoretic fiber defined by
\(\mathfrak m_y^{k+1}\mathcal O_{X_0}\).
The theorem on formal functions \cite[Section~6, Hauptsatz~IIa]{Grauert1960} gives
\[
  0=\widehat{(R^1f_*\mathcal O_{X_0})_y}
    \simeq\varprojlim_k H^1(F_k,\mathcal O_{F_k}).
\]
All these spaces have support of dimension at most one, so the kernels
of the surjections
\(\mathcal O_{F_{k+1}}\to\mathcal O_{F_k}\) have
vanishing second cohomology. Hence the induced transition maps on
\(H^1\) are surjective, and it follows that each \(H^1(F_k,\mathcal O_{F_k})\)
vanishes. Passing to the reduced fiber \(F_{\mathrm{red}}\)
gives
\[
  H^1(F_{\mathrm{red}},\mathcal O_{F_{\mathrm{red}}})=0.
\]
This is the same formal-functions argument used for rational
singularities in \cite[proof of Theorem~2.2]{Stepanov2006}.

If \(C\) is an irreducible component of \(F_{\mathrm{red}}\),
the surjection
\(\mathcal O_{F_{\mathrm{red}}}\to\mathcal O_C\) again induces a
surjection on \(H^1\). Thus \(H^1(C,\mathcal O_C)=0\).
The normalization sequence of the reduced irreducible curve
\(C\) now shows that its normalization has genus zero and that
all its singularity defects vanish. A normal complex curve is smooth, so \(C\simeq\mathbb P^1\).
\end{proof}

\subsection{Deformation of contracted curves}\label{sec:curve-deformation}

\begin{lemma}\label{lem:deform}
Let \(C\subset X_0\) be an irreducible curve contracted by \(f\).
There exists a nonconstant analytic family of smooth rational
curves \(C_s\subset X_0\), with \(C_0=C\), such that every \(C_s\)
is contracted by \(f\). The points \(f(C_s)\) cannot all be equal.
\end{lemma}

\begin{proof}
By Lemma~\ref{lem:rational-fibers}, \(C\simeq\mathbb P^1\), and
by crepancy \(K_{X_0}\cdot C=0\).
Since \(X_0\) is a fiber of a submersion over a disc,
\[
  K_{\mathcal X}|_{X_0}\simeq K_{X_0}.
\]
The normal bundle of $C$ in the fourfold \(\mathcal X\) has
rank three and degree
\[
  \deg N_{C/\mathcal X}
   =\deg(T_{\mathcal X}|_C)-\deg T_C
   =-K_{\mathcal X}\cdot C-2=-2.
\]
Riemann--Roch on \(\mathbb P^1\) gives
\begin{equation}\label{eq:normal-euler}
  \chi(N_{C/\mathcal X})=3-2=1.
\end{equation}
Embedded deformation theory has tangent space
\(H^0(C,N_{C/\mathcal X})\) and obstruction space
\(H^1(C,N_{C/\mathcal X})\); see \cite[Section~5, Theorem~5.2 and its proof]{Manetti2007}. This local calculation depends only on a neighborhood of the compact curve $C$, so it applies to the total space $\mathcal X$.
Consequently,
\begin{equation}\label{eq:douady-dimension}
  \dim_{[C]}\Douady(\mathcal X)
       \ge h^0(N_{C/\mathcal X})-h^1(N_{C/\mathcal X})=1.
\end{equation}
This is a lower bound on the dimension of the analytic germ,
and not an assertion that its obstruction space vanishes.
Choose a nonconstant reduced curve germ in that Douady space
through \([C]\), and restrict to a sufficiently small
neighborhood so that its members \(C_s\) are smooth connected
rational curves.

Every holomorphic map from a compact connected curve to
\(\Delta\) is constant. Hence each \(C_s\) lies in a single
fiber \(X_{t(s)}\), and the degree of the relative canonical bundle
is constant in this family:
\[
  K_{\mathcal X/\Delta}\cdot C_s=0.
\]
If \(t(s)\ne0\), ampleness of \(K_{X_{t(s)}}\) contradicts this
equality. Therefore \(C_s\subset X_0\) for every \(s\).

Since \(mK_{X_0}=f^*(mK_{Y_0})\) and \(mK_{Y_0}\) is ample, the equality
\(K_{X_0}\cdot C_s=0\) implies that \(f(C_s)\) is a point.
If all these points were the same point \(y\), the curves
\(C_s\) would be irreducible components of the fixed
one-dimensional fiber \(f^{-1}(y)\).
There are only finitely many such components, which is
incompatible with the nonconstant family of reduced embedded
curves. Thus the image points vary.
\end{proof}

Let \(Z\subset Y_0\) be the set where \(f\) is not an isomorphism,
which has dimension at most one. After discarding the finitely many exceptional and intersection points
of these image curves, there is a finite subset
\(\Sigma\subset Z\) such that each one-dimensional stratum of
\(Z\setminus\Sigma\) is smooth and the transverse singularity
type is a fixed Du Val type. Over such a stratum, \(f\)
induces the minimal resolution on a general transverse
surface section. To see this, choose a local coordinate along the smooth image curve
and extend it to the ambient germ of $Y_0$. Generic smoothness applied
to its pullback on the smooth space $X_0$ makes the corresponding
general surface section smooth after the exceptional parameter values
are removed, while adjunction and $K_{X_0}=f^*K_{Y_0}$ make its resolution crepant.
A crepant resolution of a Du Val surface is the minimal resolution.
We enlarge $\Sigma$ to include the discarded points.

\begin{corollary}\label{cor:generic-curves}
Every curve contracted by \(f\) can be deformed, through smooth
rational curves in \(X_0\) contracted by \(f\), to an irreducible
component of a fiber over \(Z\setminus\Sigma\).
\end{corollary}

\begin{proof}
Apply Lemma~\ref{lem:deform}. The varying points \(f(C_s)\)
cannot remain in the finite set \(\Sigma\), so for a general small parameter, \(f(C_s)\) lies in
\(Z\setminus\Sigma\), where \(C_s\) is one of the transverse
Du Val exceptional curves.
\end{proof}

\subsection{A line bundle positive on all contracted curves}\label{sec:relative-polarization}

\begin{proposition}\label{prop:positive-line}
There exists a line bundle \(L\) on \(X_0\) such that
\begin{equation}\label{eq:strict-relative-positive}
  L\cdot C>0
\end{equation}
for every irreducible curve \(C\) contracted by \(f\).
\end{proposition}

\begin{proof}
Let \(E_1,\ldots,E_r\) be its exceptional prime
divisors. Every contracted curve deforms to a general fiber
component by Corollary~\ref{cor:generic-curves}; in particular,
if there were no exceptional divisors, there could be no
contracted curves.

Fix one connected smooth stratum \(S\subset Z\setminus\Sigma\).
On a simply connected local cover of \(S\), label the
exceptional curves in a transverse resolution by
\[
  C_1,\ldots,C_\ell.
\]
Their intersection matrix on that smooth surface is
\[
  (C_i\cdot C_j)_{i,j}=-A,
\]
where \(A\) is the Cartan matrix of the associated ADE
root system. The matrix \(A\) is invertible and
\(A^{-1}\mathbf1\) has positive rational entries.
Put
\begin{equation}\label{eq:cartan-weights}
  a=A^{-1}\mathbf1,\qquad
  \mathbf1=(1,\ldots,1)^{\mathsf T}.
\end{equation}

Monodromy along \(S\) permutes the curves by Dynkin diagram
automorphisms. Such permutations preserve \(A\) and
\(\mathbf1\), and therefore preserve the unique vector \(a\)
in \eqref{eq:cartan-weights}.
Each exceptional prime divisor dominating \(\overline S\)
corresponds to an orbit of this permutation action.
We may consequently assign to that global divisor the common
coefficient \(a_i\) of its orbit.

Perform this construction for each irreducible curve in the
exceptional image. It defines a rational divisor
\[
  G=\sum_{i=1}^r\gamma_iE_i
\]
on \(X_0\). On a general transverse surface over any of the
chosen strata, its restriction is \(\sum_i a_iC_i\).
Thus, for every general fiber component \(C_j\),
\[
  (-G)\cdot C_j
     =-\sum_i a_i(C_i\cdot C_j)
     =\sum_i a_iA_{ij}=1.
\]
Exceptional divisors lying over different image curves do not
meet these general fibers.

Choose a positive integer \(N\) clearing all denominators of
the coefficients of \(G\), and set
\[
  L=\mathcal O_{X_0}(-NG).
\]
The divisor \(NG\) is Cartier because \(X_0\) is smooth, and
\[
  L\cdot C_j=N>0
\]
on every general contracted fiber component.

Finally let \(C\) be any contracted irreducible curve.
By Corollary~\ref{cor:generic-curves}, it belongs to a local
smooth family \(C_s\) whose general member is such a fiber
component. The degree of the fixed line bundle \(L\) is
constant in that family, so
\[
  L\cdot C=L\cdot C_s=N>0
\]
for a general small \(s\). This proves \eqref{eq:strict-relative-positive}.
\end{proof}

\subsection{Projectivity}\label{sec:projectivity}

We use Koll\'ar's projectivity criterion for threefolds.

\begin{proposition}[
{\cite[Corollary~5.1.5]{Kollar1991}}]\label{prop:kollar-criterion}
A proper normal \(\mathbb Q\)-factorial algebraic space \(V\)
of dimension three over \(\mathbb C\) is projective if and only
if some line bundle \(H\) on \(V\) satisfies
\[
  H\cdot C>0
  \qquad\text{for every irreducible curve }C\subset V.
\]
\end{proposition}

\begin{theorem}
\label{thm:main}
Let $\pi\colon\mathcal X\to\Delta$ be a proper holomorphic submersion
with connected compact three-dimensional fibers. If $K_{X_t}$ is ample
for every $t\ne0$, then $X_0$ is projective and $K_{X_0}$ is big and
semiample. Its canonical morphism
\[
 f\colon X_0\longrightarrow Y_0=\Proj R(X_0,K_{X_0})
\]
is crepant, all its fibers have dimension at most one, and every
exceptional prime divisor maps onto a curve. Every irreducible
contracted curve is a smooth rational curve and deforms in $X_0$ to a
general transverse Du Val exceptional curve.
\end{theorem}
\begin{proof}
Proposition~\ref{prop:semiample} gives the semiample canonical
bundle and the crepant canonical morphism
\[
  f:X_0\longrightarrow Y_0.
\]
Let \(L\) be the line bundle furnished by
Proposition~\ref{prop:positive-line}. Its degree is positive on
every irreducible curve contracted by \(f\).

The compact Moishezon threefold \(X_0\) is the analytification
of a proper algebraic space over \(\mathbb C\); see the opening
paragraph of \cite[Section~5]{Kollar1991}. By GAGA, \(f\) and
\(L\) have algebraic counterparts, denoted by the same symbols.
By Proposition~\ref{prop:fiber-dimension}, every fiber of \(f\)
has dimension at most one. Positivity of \(L\) on its irreducible
curve components therefore makes the restriction of \(L\) to
each fiber ample, by the Nakai--Moishezon criterion
\cite[Theorem~5.1.1]{Kollar1991}; nilpotent structure does not
affect ampleness. The fiberwise ampleness criterion
\cite[Tag~0D3A]{StacksProject} now makes \(L\) relatively ample.
In particular, for some integer \(a>0\), the evaluation map
\[
  f^*f_*(L^{\otimes a})\longrightarrow L^{\otimes a}
\]
is surjective.

Choose an ample line bundle \(A\) on the projective variety \(Y_0\).
For \(b\) sufficiently large, the coherent sheaf
\(f_*(L^{\otimes a})\otimes A^{\otimes b}\) is globally generated.
Consequently,
\[
  M=L^{\otimes a}\otimes f^*A^{\otimes b}
\]
is globally generated and has nonnegative degree on every
irreducible curve in \(X_0\). Set \(H=M\otimes f^*A\).
If \(C\) is contracted by \(f\), then
\(H\cdot C=a(L\cdot C)>0\). Otherwise, the projection formula gives
\[
  H\cdot C=M\cdot C+A\cdot f_*C>0.
\]
Thus \(H\) satisfies the absolute positivity hypothesis of
Proposition~\ref{prop:kollar-criterion}. Since \(X_0\) is smooth,
it is normal and \(\mathbb Q\)-factorial; that proposition shows
that \(X_0\) is projective.
Bigness was proved in
Proposition~\ref{prop:bigness}; semiampleness and crepancy are
Proposition~\ref{prop:semiample}. Proposition~\ref{prop:fiber-dimension},
Lemma~\ref{lem:rational-fibers}, and
Corollary~\ref{cor:generic-curves} supply the remaining assertions
of the theorem.
\end{proof}

\section{Counterexamples in dimension at least four}\label{sec:quartic}

\subsection{Choosing the pencil}
Let $\Gr(2,4)$ be the complex Grassmannian of complex planes in the vector space $\mathbb{C}^4$.
\begin{lemma}\label{lem:pencil}
There is a smooth quartic surface $S_0\subset\PP^3$ containing exactly one
line $\ell$, and there is a quartic polynomial $F_1$ such that the pencil
\[
 S_t=\{F_0+tF_1=0\}\subset\PP^3
\]
has the following properties after restriction to a sufficiently small
disc $\Delta$:
\begin{enumerate}
\item every $S_t$ is smooth;
\item $S_t$ contains no line for $t\ne0$;
\item after locally trivializing the universal line over $G=\Gr(2,4)$ near
$[\ell]$, the map
\begin{equation}\label{eq:coefficient-map}
 J:(G\times\Delta,([\ell],0))\longrightarrow
 (H^0(\PP^1,\OO_{\PP^1}(4)),0),\qquad
 (L,t)\longmapsto(F_0+tF_1)|_L
\end{equation}
is a local biholomorphism.
\end{enumerate}
\end{lemma}

\begin{proof}
Fix a line $\ell\subset\PP^3$. Quartic hypersurfaces containing $\ell$
form a projective space of dimension $29$. For a second line $L$ skew to
$\ell$, the condition of also containing $L$ imposes five independent
linear conditions, while skew lines vary in a four-dimensional family;
for a second line meeting $\ell$ there are four independent conditions and
a three-dimensional family of choices of $L$. In either case the
incidence locus has dimension at most $28$. Thus a general quartic
containing $\ell$ contains no other line. Bertini gives smoothness away
from $\ell$, and smoothness along $\ell$ follows by choosing the two normal
cubic derivatives with no common zero. This gives $S_0$.

Adjunction on the K3 surface $S_0$ gives
$N_{\ell/S_0}\simeq\OO_{\PP^1}(-2)$, so the normal bundle sequence is
\begin{equation}\label{eq:normal-line}
 0\longrightarrow\OO_{\PP^1}(-2)
 \longrightarrow\OO_{\PP^1}(1)^{\oplus2}
 \longrightarrow\OO_{\PP^1}(4)\longrightarrow0.
\end{equation}
The differential in the $G$-direction of \eqref{eq:coefficient-map} is the
map on global sections induced by the last arrow. It is injective, with
four-dimensional image in a five-dimensional space. Choose $F_1$ so that
$F_1|_\ell$ lies outside that image; restriction of quartic polynomials to
$\ell$ is surjective. The full differential of $J$ is then an isomorphism between two
five-dimensional vector spaces, so the holomorphic inverse function
theorem proves (iii). Notice that the normal sequence has a genuinely
one-dimensional cokernel on global sections, since
$h^1(\OO_{\PP^1}(-2))=1$. Thus the transversality condition on $F_1$
is achievable without an additional genericity hypothesis.

Smoothness of $S_t$ holds after shrinking the disc. To prove (ii), suppose
that $t_\nu\to0$, $t_\nu\ne0$, and $S_{t_\nu}$ contains a line $L_\nu$.
Compactness of $G$ yields a limiting line on $S_0$, necessarily $\ell$.
But $J$ has a unique zero near $([\ell],0)$, contradicting the existence
of $(L_\nu,t_\nu)$. A further shrinking gives the assertion.
\end{proof}

Let $\mathcal S\to\Delta$ be this pencil and put
\[
 h:\mathcal Y=\Hilb^2(\mathcal S/\Delta)\longrightarrow\Delta.
\]
The relative Hilbert scheme is a smooth projective morphism of relative
dimension four, by the smoothness theorem for Hilbert schemes of points
on surfaces \cite[Theorems~2.4 and~2.9]{Fogarty}. Each $Y_t=S_t^{[2]}$ carries the holomorphic
symplectic form associated with the K3 surface $S_t$ \cite[Section~6, Proposition~5]{Beauville}; in
particular $K_{Y_t}\simeq\OO_{Y_t}$.

\subsection{The secant line bundle}

Every length-two subscheme $\xi\subset\PP^3$ spans a unique line, which for a
nonreduced subscheme is the line determined by its tangent direction.
The construction is algebraic: evaluation of linear forms on $\xi$ gives
a rank-two quotient of $H^0(\PP^3,\OO(1))$. Consequently it defines a
morphism
\begin{equation}\label{eq:secant-map}
 \Phi:\mathcal Y\longrightarrow G\times\Delta,
 \qquad \xi\longmapsto(\langle\xi\rangle,t).
\end{equation}
Write
\[
 H=\Phi^*\pr_G^*\OO_G(1),\qquad H_t=H|_{Y_t}.
\]

\begin{lemma}\label{lem:secant}
For $t\ne0$, the map $\Phi_t:Y_t\to G$ is finite of degree six, and $H_t$
is ample. The map $\Phi_0$ is generically finite of degree six, and its only
positive-dimensional fiber is
\[
 P=\ell^{[2]}\simeq\PP^2=\Phi_0^{-1}([\ell]).
\]
The line bundle $H_0$ is big and globally generated, and $H|_P$ is trivial.
Furthermore,
\begin{equation}\label{eq:normal-central}
 N_{P/Y_0}\simeq\Omega^1_P.
\end{equation}
\end{lemma}

\begin{proof}
If $L$ is not contained in $S_t$, then $S_t\cap L$ is a degree-four
effective divisor on $L\simeq\PP^1$, and it has only finitely many degree-two
subdivisors, which are exactly the points of $\Phi_t^{-1}([L])$. Thus
$\Phi_t$ is quasi-finite when $S_t$ contains no line, and properness makes
it finite. A general line meets $S_t$ in four distinct points, giving
$\binom42=6$ length-two subschemes. The same degree calculation applies to
$t=0$, where the only remaining fiber is
$\Hilb^2(\ell)=\operatorname{Sym}^2(\PP^1)=\PP^2$.

The assertions about $H_t$ follow by pulling back the Pl\"ucker line
bundle; in particular, the pullback is big because $\Phi_0$ is generically
finite. The restriction of the symplectic form on $Y_0$ to $P$ is zero,
since $H^0(P,\Omega^2_P)=0$. As $P$ has half the dimension of $Y_0$, it is
Lagrangian, and the symplectic form identifies its normal bundle with its
cotangent bundle. This proves \eqref{eq:normal-central}.
\end{proof}

\subsection{A local model for the total space}

Set
\[
 V=H^0(\PP^1,\OO_{\PP^1}(2)),\qquad
 W=H^0(\PP^1,\OO_{\PP^1}(4)),
\]
so that $\dim V=3$ and $\dim W=5$. Let
$\mu:V\otimes V\to W$ be multiplication of binary quadratic forms, and consider
\[
 A=\Tot\bigl(\OO_{\PP(V)}(-1)\otimes V\bigr).
\]
A point of $A$ is written as $([q],q\otimes r)$, with the usual
identification $(q,r)\sim(cq,c^{-1}r)$ for $c\ne0$.

\begin{proposition}\label{prop:local-model}
A neighborhood of $P$ in $\mathcal Y$ is isomorphic to a neighborhood
of the zero section in $A$. In these coordinates the composite with the
coefficient map $J$ is
\begin{equation}\label{eq:multiplication}
 ([q],q\otimes r)\longmapsto qr\in W,
\end{equation}
and $h$ takes the form
\begin{equation}\label{eq:local-time}
 t=\tau(qr),
\end{equation}
where $\tau:(W,0)\to(\Delta,0)$ is holomorphic. If
$\lambda=d\tau_0$, the bilinear form
\begin{equation}\label{eq:pairing}
 b(q,r)=\lambda(qr)
\end{equation}
on $V\times V$ is nondegenerate. In particular,
\begin{equation}\label{eq:normal-total}
 N_{P/\mathcal Y}\simeq\OO_{\PP^2}(-1)^{\oplus3}.
\end{equation}
\end{proposition}

\begin{proof}
Use the local trivialization of the universal line in
Lemma~\ref{lem:pencil}. An effective degree-two divisor on that line is
specified by $[q]\in\PP(V)$, and it lies on the quartic surface exactly when
the restricted quartic equation is divisible by $q$. This gives the
incidence description
\[
 \{(s,[q])\in W\times\PP(V):s\in qV\}.
\]
The map $\OO_{\PP(V)}(-1)\otimes V\to W\otimes\OO_{\PP(V)}$ induced
by multiplication is injective on every fiber, since multiplication by a
nonzero quadratic is injective on quadratic forms. Its image is thus
a rank-three subbundle, and the incidence space is exactly the smooth
total space $A$. In particular, this description also works when the
divisor $q$ is nonreduced or when the residual quadratic shares a root
with $q$. By Lemma~\ref{lem:pencil}(iii), the
coordinates $s=J(L,t)$ identify this incidence description with a
neighborhood of the entire fiber $P$ in $\mathcal Y$; if necessary, take
the inverse image of a smaller neighborhood of $0$ in $W$.

The function $\tau$ is the $t$-coordinate of $J^{-1}$, which proves
\eqref{eq:local-time}. At $([q],0)$ the differential of $h$ vanishes
in tangent directions to the zero section and, in the normal direction
$q\otimes r$, equals $\lambda(qr)$. Since $h$ is a smooth morphism at
\emph{every} $[q]\in\PP(V)$, the functional $b(q,-)$ is nonzero for
every $q\ne0$. Hence the linear map $V\to V^*$, $q\mapsto b(q,-)$,
has trivial kernel and is an isomorphism, which proves that $b$ is nondegenerate. Finally, the normal bundle of the zero section in
$A$ is $\OO_{\PP(V)}(-1)\otimes V$, giving
\eqref{eq:normal-total}.
\end{proof}

\begin{remark}\label{rem:euler}
The differential of $t$ gives a surjective bundle map
\[
 \OO_{\PP(V)}(-1)\otimes V\longrightarrow\OO_{\PP(V)}.
\]
The pairing $b$ identifies $V$ with $V^*$, so its kernel is
$\Omega^1_{\PP(V)}$ by the Euler sequence. Thus the local model also
recovers \eqref{eq:normal-central} without using the symplectic form.
The function $\tau$ need not be linear; only its first derivative is used
in the submersion argument below.
\end{remark}

\subsection{The branch divisor and the canonical bundle}

Choose a positive integer $m$ and a general member
$Q\in|\OO_G(2m)|$ such that $[\ell]\notin Q$. The pullback linear
system on $Y_0$ is basepoint-free, so Bertini's theorem allows us to choose
$Q$ in such a way that
\[
 B_0=\Phi_0^{-1}(Q)
\]
is smooth. It is disjoint from $P$. Let
\[
 B=\Phi^{-1}(Q\times\Delta)\in|H^{\otimes2m}|.
\]
The nonsmooth locus of the relative divisor $B\to\Delta$ is a closed
subset of $B$. Since $\mathcal Y\to\Delta$ is proper, the image of this
closed locus is closed in $\Delta$, and it does not contain $0$ because
$B_0$ is smooth. After shrinking $\Delta$, therefore, $B\to\Delta$
is smooth. By choosing $Q$ general we may also assume that $Q$ itself
is a smooth divisor on $G$; the relevant requirement for the cover is
the smoothness of its pullback. Fix such a choice henceforth.

Set $L=H^{\otimes m}$, and let $s\in H^0(\mathcal Y,L^{\otimes2})$
define $B$. Form the double cover
\begin{equation}\label{eq:double-cover}
 f:\mathcal Z=\Spec_{\mathcal Y}
       \bigl(\OO_{\mathcal Y}\oplus L^{-1}\bigr)
       \longrightarrow\mathcal Y,
\end{equation}
where multiplication $L^{-2}\to\OO_{\mathcal Y}$ is given by $s$.
Locally the cover has equation $u^2=s$.

\begin{proposition}\label{prop:cover}
The morphism $\rho=h\circ f:\mathcal Z\to\Delta$ is a smooth projective
morphism with connected four-dimensional fibers. After shrinking the disc
if necessary,
\begin{equation}\label{eq:cover-canonical}
 K_{\mathcal Z/\Delta}\simeq f^*H^{\otimes m}.
\end{equation}
Consequently $K_{Z_t}$ is ample for every $t\ne0$, while $K_{Z_0}$ is big
and globally generated. Moreover,
\begin{equation}\label{eq:planes}
 f_0^{-1}(P)=P_+\sqcup P_-,\qquad P_\pm\simeq\PP^2,
\end{equation}
and a neighborhood of each $P_\pm$ in $\mathcal Z$, together with the
map to $\Delta$, is a copy of the local model in
Proposition~\ref{prop:local-model}.
\end{proposition}

\begin{proof}
Since $B\to\Delta$ is smooth, the equation $u^2=s$ defines a smooth
relative double cover, including along the ramification divisor. The
cover is finite over the projective family $\mathcal Y$, hence projective
over $\Delta$. The divisor $B_t$ is nonempty on every fiber because
$\Phi_t$ is surjective. If a smooth fiber $Z_t$ were disconnected, its
two components would have degree one over the normal connected manifold
$Y_t$, and each would therefore map isomorphically to $Y_t$, contradicting
the nonemptiness of the branch divisor. Thus the fibers are connected.

Every $K_{Y_t}$ is trivial and $h^0(Y_t,K_{Y_t})=1$, so Grauert's theorem and
the fiberwise evaluation map give
$K_{\mathcal Y/\Delta}\simeq h^*h_*K_{\mathcal Y/\Delta}$.
The rank-one bundle on the disc is trivial, so
$K_{\mathcal Y/\Delta}\simeq\OO_{\mathcal Y}$. Adjunction for the
equation $u^2=s$ in $\Tot(L)$ gives the double-cover formula \cite[(1.2)]{LanteriStruppa}
\[
 K_{\mathcal Z/\Delta}
 \simeq f^*(K_{\mathcal Y/\Delta}\otimes L)
 \simeq f^*H^{\otimes m}.
\]
Lemma~\ref{lem:secant} proves the stated positivity properties.

Because $[\ell]\notin Q$, the branch divisor is disjoint from $P$.
Take a sufficiently small (topological) tubular neighborhood of $P$
in the smooth total space $\mathcal Y$ which is disjoint from $B$.
It deformation-retracts onto the simply connected space $P\simeq\PP^2$.
The restricted unramified degree-two covering is therefore the disjoint
union of two degree-one coverings, hence of two biholomorphic copies of
this neighborhood. This proves \eqref{eq:planes}, and shows in addition
that each copy carries the same map to $\Delta$ as in
Proposition~\ref{prop:local-model}.
\end{proof}

\subsection{The homology relation before the flop}

\begin{lemma}\label{lem:homology}
Let $C_+\subset P_+$ and $C_-\subset P_-$ be projective lines. Then
\begin{equation}\label{eq:equal-lines}
 [C_+]=[C_-]\ne0\quad\text{in }H_2(Z_0,\R).
\end{equation}
\end{lemma}

\begin{proof}
For $t\ne0$ the branch divisor $B_t\in|H_t^{\otimes2m}|$ is ample, and the
Lefschetz theorem for cyclic covers
\cite[Proposition 1.11]{LanteriStruppa} gives an isomorphism
\[
 (f_t)_*:H_2(Z_t,\R)\xrightarrow{\sim}H_2(Y_t,\R);
\]
the theorem in fact gives integral homology isomorphisms in degrees less
than the complex dimension. Equivalently, $f_t^*$ is an isomorphism on
$H^2$ with real coefficients.

Because $f:\mathcal Z\to\mathcal Y$ is a morphism over $\Delta$
between two proper holomorphic submersions, its restrictions induce a
morphism $R^2h_*\R\to R^2\rho_*\R$ of locally constant sheaves on
\emph{the whole disc}. In locally constant trivializations this morphism
is represented by one fixed linear map, so the isomorphism established on
$\Delta^*$ implies the same statement at $0$:
\[
 f_0^*:H^2(Y_0,\R)\xrightarrow{\sim}H^2(Z_0,\R),
\]
and, by duality over $\R$, $(f_0)_*$ is also an isomorphism on $H_2$.
Both $C_+$ and $C_-$ map isomorphically to lines in $P$, whose homology
classes coincide, which proves their equality in $Z_0$. The common class
is nonzero because $Z_0$ is projective and an ample class has positive
degree on either line.
\end{proof}

\subsection{An elementary transformation over the disc}\label{sec:flop}

Blow up only the plane $P_+$ in the total space:
\[
 p:\widetilde{\mathcal Z}=\Bl_{P_+}\mathcal Z\longrightarrow\mathcal Z.
\]
Its exceptional divisor will be denoted by $D$. By
Proposition~\ref{prop:cover},
\begin{equation}\label{eq:exceptional-total}
 D\simeq\PP(V)\times\PP(V),\qquad
 \OO_D(D)\simeq\OO_{\PP(V)\times\PP(V)}(-1,-1).
\end{equation}
The first projection is the restriction of $p$.

\begin{proposition}\label{prop:flop-family}
There is a proper holomorphic map
$q:\widetilde{\mathcal Z}\to\mathcal X$ to a smooth complex fivefold
which contracts $D$ along its second projection and is an isomorphism
away from $D$. The induced map $\pi:\mathcal X\to\Delta$ is a proper
holomorphic submersion, and its restriction over $\Delta^*$ is isomorphic to
$\mathcal Z|_{\Delta^*}\to\Delta^*$.

Writing $P_+^\vee=q(D)\simeq\PP^2$, the central fibers fit into a diagram
\begin{equation}\label{eq:central-diagram}
 Z_0\xleftarrow{\ p_0\ }\widetilde Z_0
       \xrightarrow{\ q_0\ }X_0,
\end{equation}
where $p_0$ and $q_0$ are the blow-ups of $P_+$ and $P_+^\vee$,
respectively. Their common exceptional divisor is
\begin{equation}\label{eq:incidence}
 E=\{([q],[r])\in\PP(V)\times\PP(V):b(q,r)=0\},
 \qquad \OO_E(E)\simeq\OO_E(-1,-1).
\end{equation}
The birational transformation is an isomorphism in a neighborhood of
$P_-\subset Z_0$.
\end{proposition}

\begin{proof}
The normal line bundle in \eqref{eq:exceptional-total} restricts to
$\OO_{\PP^2}(-1)$ on every fiber of the second projection. Since the
second projection $D\to\PP(V)$ is a smooth $\PP^2$-bundle and its normal
line bundle has the required tautological restriction, the
Fujiki--Nakano inverse blow-up criterion \cite[Section~2]{FujikiNakano} produces
a proper analytic contraction $q$ to a smooth complex manifold,
contracting precisely these fibers. Away from $D$ the map is an
isomorphism. The contraction is global, but analytic: no projectivity
of $\mathcal X$ is assumed or required. The same blow-up and
blow-down procedure appears in \cite[Section~3, proof of Theorem~3.4]{Huybrechts}.

For completeness, we check the map to the disc in the explicit
neighborhood of $P_+$. Introduce the rank-one tensor cone
\[
 \Sigma=\{T\in V\otimes V:\operatorname{rank}T\le1\}.
\]
The two small resolutions of this cone are
\begin{align*}
 A&=\Tot\bigl(\OO_{\PP(V)}(-1)\otimes V\bigr),
       &([q],q\otimes r)&\longmapsto q\otimes r,\\
 A'&=\Tot\bigl(V\otimes\OO_{\PP(V)}(-1)\bigr),
       &([r],q\otimes r)&\longmapsto q\otimes r.
\end{align*}
The blow-up of the zero section in either resolution is
$\Tot(\OO_{\PP(V)\times\PP(V)}(-1,-1))$. Thus the second contraction
is given locally by replacing $A$ with $A'$.

On the cone the function defining the base parameter is
\[
 T\longmapsto\tau(\mu(T)).
\]
It consequently defines the same holomorphic base map on both
resolutions. On the zero section of $A'$, its differential in the fiber
direction $q$ is $b(q,r)$. Nondegeneracy of $b$ shows that this is nonzero
for every $[r]\in\PP(V)$, so the descended map $\pi$ is a submersion
along $P_+^\vee$. Away from that plane it is identified with the original
submersion $\rho$.

The map $\pi$ is proper. Indeed, for a compact subset $K\subset\Delta$,
the set $(\rho\circ p)^{-1}(K)$ is compact, and its image under the
surjective map $q$ is exactly $\pi^{-1}(K)$. The construction is an
isomorphism over $\Delta^*$, because its center lies in the central fiber.

It remains to identify that fiber and its exceptional divisor. In the
common blow-up, write the tensor as $u(q\otimes r)$, where $u$ is a local
coordinate in the tautological line. Taylor expansion gives
\[
 \tau(uqr)=u\,b(q,r)+O(u^2).
\]
As $h$ (and hence $\rho$) is a submersion along $P_+$, the pullback
of its central-fiber equation has vanishing order exactly one on the
blow-up center. Thus the total inverse image of $0$ is the reduced
union of $D$ and the strict transform $\widetilde Z_0$, the latter
having, in this chart, the equation
\[
 b(q,r)+O(u)=0.
\]
It meets $D$ transversely in the smooth incidence divisor
$E=\{b(q,r)=0\}$. Nondegeneracy identifies the second projective plane
with the dual of the first, so each projection of $E$ is a
$\PP^1$-bundle. The strict transform $\widetilde Z_0$ is the blow-up
of $Z_0$ along $P_+$, while its contraction onto the now smooth $X_0$
is the blow-up of $X_0$ along $P_+^\vee$; the inverse-blow-up assertion
may also be seen directly from the local rank-one tensor model.
Restriction of \eqref{eq:exceptional-total} gives the claimed normal
bundle in \eqref{eq:incidence}. Since $P_-$ is disjoint from the center,
it is untouched by the construction.
\end{proof}

\subsection{The obstruction to a K\"ahler form}\label{sec:obstruction}

We record the change of degrees under this elementary transformation.

\begin{lemma}\label{lem:sign}
The common complement of the planes in \eqref{eq:central-diagram}
induces an isomorphism
\begin{equation}\label{eq:H2-identification}
 H^2(Z_0,\R)\simeq H^2(X_0,\R).
\end{equation}
Let $\alpha$ and $\alpha'$ be corresponding classes. For projective lines
$C_+\subset P_+$ and $C_+^\vee\subset P_+^\vee$ one has
\begin{equation}\label{eq:sign-change}
 \int_{C_+^\vee}\alpha'=-\int_{C_+}\alpha.
\end{equation}
Moreover, for a line $C_-\subset P_-$,
$\int_{C_-}\alpha'=\int_{C_-}\alpha$.
\end{lemma}

\begin{proof}
The planes have real codimension four. The Thom isomorphism and the
long exact sequence with supports show that removal of either plane
induces an isomorphism on $H^2$: the support groups in degrees two and
three vanish. This proves \eqref{eq:H2-identification}.

The class $p_0^*\alpha-q_0^*\alpha'$ vanishes on
$\widetilde Z_0\setminus E$, so the exact sequence with supports along the
smooth connected divisor $E$ gives
\begin{equation}\label{eq:exceptional-class}
 p_0^*\alpha-q_0^*\alpha'=a[E]
\end{equation}
for some $a\in\R$, where $[E]=c_1(\OO_{\widetilde Z_0}(E))$.
Put $d=\int_{C_+}\alpha$ and $d'=\int_{C_+^\vee}\alpha'$.
A fiber $F_q\simeq\PP^1$ of $q_0|_E$ maps under $p_0$ isomorphically
to a line in $P_+$, while a fiber $F_p$ of $p_0|_E$ maps under $q_0$
isomorphically to a line in $P_+^\vee$. By \eqref{eq:incidence},
\[
 E\cdot F_p=E\cdot F_q=-1.
\]
Evaluating \eqref{eq:exceptional-class} on these two fibers yields
$d=-a$ and $-d'=-a$, hence $d'=-d$. The assertion for $C_-$ follows
because this curve lies in the common complement.
\end{proof}

\begin{proposition}\label{prop:nonkahler}
The central fiber $X_0$ is Moishezon and not K\"ahler. More precisely,
\begin{equation}\label{eq:opposite-lines}
 [C_+^\vee]+[C_-]=0\quad\text{in }H_2(X_0,\R),
\end{equation}
and neither summand is zero.
\end{proposition}

\begin{proof}
The manifold $X_0$ is bimeromorphic to the projective manifold $Z_0$, so
it is Moishezon. For an arbitrary $\alpha'\in H^2(X_0,\R)$, let
$\alpha\in H^2(Z_0,\R)$ correspond under
\eqref{eq:H2-identification}. Lemmas~\ref{lem:homology} and
\ref{lem:sign} give
\[
 \int_{C_+^\vee}\alpha'
 =-\int_{C_+}\alpha
 =-\int_{C_-}\alpha
 =-\int_{C_-}\alpha'.
\]
The perfect pairing between real $H^2$ and $H_2$ proves
\eqref{eq:opposite-lines}. Taking $\alpha$ to be an ample class on $Z_0$,
its evaluation on $C_-$ is positive, which also proves that both classes
in \eqref{eq:opposite-lines} are nonzero.

If $\omega$ were a K\"ahler form on $X_0$, its integrals over the two
smooth complex curves would both be positive, whereas
\eqref{eq:opposite-lines} would imply
\[
 0=\int_{C_+^\vee}\omega+\int_{C_-}\omega>0.
\]
This contradiction proves that $X_0$ is not K\"ahler.
\end{proof}

\begin{proof}[Proof of Theorem~\ref{thm:KE-main}]
Propositions~\ref{prop:cover}, \ref{prop:flop-family}, and
\ref{prop:nonkahler} prove the assertion for $n=4$. For $n>4$, let $T$
be a product of $n-4$ smooth projective curves of genus at least two,
and form the product family
\[
 \mathcal X\times T\longrightarrow\Delta.
\]
It is a proper holomorphic submersion with connected compact fibers of
dimension $n$. For $t\ne0$, the canonical bundle of $X_t\times T$ is
the exterior tensor product of two ample line bundles, hence ample.
The central fiber is Moishezon because $X_0$ is bimeromorphic to a
projective manifold and $T$ is projective. If $X_0\times T$ were
K\"ahler, restricting a K\"ahler form to $X_0\times\{u\}$ would make
$X_0$ K\"ahler, a contradiction.
\end{proof}

\begin{remark}As suggested by Prof. Koll\'ar, we can replace the initial quartic by the Fermat surface and use the pencil
\begin{equation}\label{eq:fermat-pencil}
 S_t=\{x_0^4+x_1^4+x_2^4+x_3^4+t x_0x_1x_2x_3=0\}
 \subset\PP^3.
\end{equation}
The central surface has exactly forty-eight lines: for each of the three
partitions $\{ij\mid kl\}$ of $\{0,1,2,3\}$ into pairs, they are
\[
 x_i=\alpha x_j,\qquad x_k=\beta x_l,
 \qquad\alpha^4=\beta^4=-1
\]
(see \cite[Section~3]{ShimadaShioda2017}). On such a line, write
$(x_i,x_j,x_k,x_l)=(\alpha u,u,\beta v,v)$. The image of the
Grassmannian derivative of the coefficient map
\eqref{eq:coefficient-map} is
$\langle u^4,u^3v,uv^3,v^4\rangle$, whereas its parameter derivative is
$\alpha\beta u^2v^2$. Thus the coefficient map is locally biholomorphic
at all forty-eight lines. The compactness argument in
Lemma~\ref{lem:pencil} shows that, after shrinking the disc, every $S_t$
is smooth and no $S_t$ with $t\ne0$ contains a line.

Let $\ell_1,\ldots,\ell_{48}$ be these lines and set
$P_i=\ell_i^{[2]}\subset Y_0=S_0^{[2]}$. These planes are pairwise
disjoint, even when the original lines meet: they are fibers of
$\Phi_0$ over the distinct points $[\ell_i]\in G$. The local model in
Proposition~\ref{prop:local-model} applies at each plane. Choose
$Q\in|\OO_G(2m)|$ avoiding all these points, with smooth pullback branch
divisor as above. The double cover then contains ninety-six disjoint
planes $P_{i,+},P_{i,-}$, and Lemma~\ref{lem:homology} gives
$[C_{i,+}]=[C_{i,-}]\ne0$ for lines in each pair. For any subset
\[
 A\subset I:=\{1,\ldots,48\}\times\{+,-\},
\]
perform the relative elementary transformations at the planes indexed
by $A$, and denote the resulting family by $\mathcal X_A\to\Delta$.
Disjointness allows these operations to be performed independently.
Propositions~\ref{prop:flop-family} and~\ref{prop:nonkahler} apply
unchanged: the family is a proper holomorphic submersion, its punctured
fibers are canonically polarized, and its central fiber is Moishezon
and non-K\"ahler whenever
\begin{equation}\label{eq:fermat-split}
 \#\bigl(A\cap\{(i,+),(i,-)\}\bigr)=1
 \quad\text{for at least one }i.
\end{equation}

For fixed $m$ and a \emph{very general} choice of $Q$, there are $2^{94}$ different isomorphism
classes of these families.
\end{remark}

\end{document}